\documentclass[10pt,reqno]{amsart}
\usepackage{amsmath,amsfonts,amsthm,amscd,amssymb,graphicx}
\usepackage[utf8]{inputenc}
\usepackage{amsbsy}
\usepackage{graphicx}
\usepackage{hyperref}
\usepackage[margin=1in]{geometry}
\usepackage{subcaption}
\usepackage{enumitem}
\usepackage{float}
\usepackage{mathtools}
\usepackage{dsfont}
\usepackage{upgreek}
\usepackage{stmaryrd}

\usepackage{color}
\definecolor{myblue}{rgb}{.8, .8, 1}

\usepackage{amsmath}
\usepackage{empheq}
\usepackage{comment}

\newlength\mytemplen
\newsavebox\mytempbox

\makeatletter
\newcommand\mybluebox{%
    \@ifnextchar[
       {\@mybluebox}%
       {\@mybluebox[0pt]}}

\def\@mybluebox[#1]{%
    \@ifnextchar[
       {\@@mybluebox[#1]}%
       {\@@mybluebox[#1][0pt]}}

\def\@@mybluebox[#1][#2]#3{
    \sbox\mytempbox{#3}%
    \mytemplen\ht\mytempbox
    \advance\mytemplen #1\relax
    \ht\mytempbox\mytemplen
    \mytemplen\dp\mytempbox
    \advance\mytemplen #2\relax
    \dp\mytempbox\mytemplen
    \colorbox{myblue}{\hspace{1em}\usebox{\mytempbox}\hspace{1em}}}

\makeatother

\usepackage{tikz}
\usetikzlibrary{decorations.shapes}
\tikzset{paint/.style={ draw=#1!50!black, fill=#1!50 },
    decorate with/.style=
    {decorate,decoration={shape backgrounds,shape=#1,shape size=2mm}}}
\usepackage{todonotes}

\definecolor{skyblue}{rgb}{0.85,0.85,1}

\usepackage{accents}

\newtheorem{theorem}{Theorem}
\newtheorem{prop}[theorem]{Proposition}
\newtheorem{cor}[theorem]{Corollary}
\newtheorem{define}[theorem]{Definition}
\newtheorem{lemma}[theorem]{Lemma}
\newtheorem{rem}[theorem]{Remark}
\newtheorem{assumption}[theorem]{Assumption}

\newcommand{\bbR}{\mathbb{R}}             

\newcommand{\p}{\partial}				

\newcommand{\eps}{\epsilon}

\newcommand{\pa}{\partial}

\newcommand{\eqdef}{\vcentcolon =}
\newcommand{\defeq}{= \vcentcolon}

\newcommand{\norm}[1]{\left\lVert#1\right\rVert}

\numberwithin{equation}{section}
\numberwithin{theorem}{section}

\begin{document}
\title{Nodal Deficiency of Neumann Eigenfunctions on a Symmetric Dumbbell Domain}
\author{Thomas Beck and Andrew Lyons}

\address{Department of Mathematics, John Mulcahy Hall\\
 Fordham University }
\email{tbeck7@fordham.edu}

\address{Department of Mathematics, Mathematics \& Science Center\\
 Emory University}
\email{ahlyons@emory.edu}

\begin{abstract}
We study the nodal deficiency of pairs of Neumann eigenfunctions defined over symmetric dumbbell domains. As the width of the connecting neck shrinks, these eigenfunctions converge to Neumann eigenfunctions defined over the ends of the dumbbell, together with a one-dimensional Sturm-Liouville solution in the neck. In this limit, the corresponding eigenvalues become degenerate, with multiplicity two. The nodal deficiency, defined as the difference between the eigenvalue index and the nodal domain count, is known by the Courant nodal domain theorem to be nonnegative. We show that, for small neck widths, the nodal deficiencies of the dumbbell eigenfunctions are no smaller than the nodal deficiencies of the limiting eigenfunctions in the ends, and we provide conditions under which equality is achieved. As a consequence, we establish a criterion for identifying eigenfunctions of zero nodal deficiency for the dumbbell domain.
\end{abstract}

\maketitle

\section{Introduction} \label{sec:Intro}

\quad We study the nodal sets of eigenfunctions of the Neumann Laplacian corresponding to a perturbation of an eigenvalue of multiplicity $2$. The eigenfunctions are defined over a symmetric dumbbell domain, interpreted as a perturbation of two disjoint by a thin neck connecting them. For a general open, bounded domain $\Omega\subset \mathbb{R}^2$ with piecewise smooth boundary, we consider a Neumann eigenfunction $u$ satisfying
\begin{equation}\label{eq:PDE}
    \begin{cases}
        \left(\Delta+\lambda\right)u=0, & \textrm{in } \Omega \\ \p_\nu u=0, & \textrm{on } \p\Omega
    \end{cases}
\end{equation}
with eigenvalue $\lambda$. Here $\p_\nu$ denotes the normal derivative along the boundary $\p\Omega$. The \textit{nodal set} of $u$ is then the closure of the level set $\{(x,y)\in\Omega:u(x,y)=0\}$, partitioning $\Omega$ into connected regions, called \textit{nodal domains}, where $u$ is either strictly positive or negative. We say that the nodal set of $u$ contains no crossings if it consists of disjoint closed curves in the closure $\overline{\Omega}$. Laplacian eigenfunctions model vibrational modes in various media, and the nodal set of $u$ describes points that remain fixed under a vibration with energy given by $\lambda$. The nodal set of an eigenfunction depends heavily on the geometry of $\Omega$ and is sensitive to domain perturbations. For this reason, their study is a classical and active area of research \cite{C76,Z17}. 

\quad Let $\llbracket u,\Omega\rrbracket$ denote the number of nodal domains of the function $u$ satisfying (\ref{eq:PDE}) in its domain $\Omega$. The Courant nodal domain  theorem \cite{CH62} asserts that $\llbracket u,\Omega\rrbracket$ cannot exceed the minimal index of the eigenvalue $\lambda$ in the Neumann spectrum of $\Omega$. If an eigenfunction satisfies $\llbracket u,\Omega\rrbracket =\textrm{index}(\lambda,\Omega)\eqdef \min\{j:\lambda=\lambda_j\}$, then we label both the eigenfunction $u$ and the corresponding eigenvalue $\lambda$ as \textit{Courant sharp}. More generally, the nodal deficiency
\begin{equation}\label{eq:NodDef}
    \upvartheta(u,\Omega)\eqdef \textrm{index}(\lambda,\Omega)-\llbracket u,\Omega\rrbracket \geq 0
\end{equation}
serves as a measure of Courant sharpness. Over domains of dimension $1$, the eigenfunctions and eigenvalues of (\ref{eq:PDE}) are all Courant sharp; the nodal set of the $j$th eigenfunction, labeled according to the  position of the eigenvalue in the spectrum, consists of $j-1$ interior points and separates exactly $j$ nodal domains. In particular, this property is stable under perturbations of the interval domain. In contrast, there are only finitely many Courant sharp eigenfunctions defined over domains of higher dimension \cite{deponti23pleijel,lena-pleijel,P56,P09}. Generically, nodal sets have no crossings, and the nodal domain count should be stable under domain perturbations \cite{U76}. However, in some cases with nodal set crossings, the number of nodal domains is unstable and may decrease (but not increase) under generic domain deformations \cite{BGM,L25,MMS25}. Hence, the nodal deficiency is a difficult quantity to compute over a general domain, where the eigenfunctions and eigenvalues may not be explicit. Another approach developed to understand nodal deficiency involves writing it in terms of a Morse index for Dirichlet-to-Neumann maps of the eigenfunction nodal domains \cite{berkolaiko2022-1,berkolaiko2022-2,berkolaiko2019,berkolaiko2012}. 

\quad We consider eigenfunctions of (\ref{eq:PDE}) defined over a symmetric dumbbell domain, consisting of two disjoint, bounded, planar domains connected by a thin neck, precisely defined below in Section \ref{sec:Domain}. The stochastic block model for a data set generates a graph with distinct clusters, or sub-graphs that are weakly connected to one another, and dumbbell domains, or more generally chain domains consisting of a finite number of sub-domains joined by thin necks, serve as a continuum model for large data sets with this structure \cite{singer2017spectral,GS18}. The nodal domains associated to low-energy eigenvalues provide a natural partition that is indicative of community structure. In \cite[Theorem $1.1$]{BCM-pleijel}, Beck, Canzani, and Marzuola provided an upper bound, independent of the structure of the thin connecting necks, for the number of nodal domains produced by an eigenfunction of (\ref{eq:PDE}) over a chain domain. They accomplished this by identifying several geometric constants of the disjoint planar domains that control the eigenvalues of Courant sharp Neumann eigenfunctions, and then by applying the Faber-Krahn Theorem to obtain lower bounds for the area of nodal domains in terms of these geometric constants. Their results show that the number of Courant sharp eigenfunctions of chain domains is finite, with a bound on their number independent of the width of the necks. Their work also indicates how the qualitative nodal behavior of Courant sharp eigenfunctions defined over a chain domain is unaffected by the widths of the connecting necks.

\quad The symmetry of our dumbbell domain implies that, as the neck degenerates, the limiting spectrum contains eigenvalues with even multiplicity. To understand how the nodal deficiency varies with respect to the neck width, we first establish the ordering of eigenvalue branches that share a limit. In \cite[Theorem $2.5$]{A95}, Arrieta established the first-order variation for Neumann eigenvalues of domains connected by a thin neck. However, this work relied on the assumption that the limiting eigenvalue was simple, which is broken by the symmetry of our domain. Following techniques used in \cite{A95, J93}, we establish analogous estimates for limiting eigenvalues on the separated domain with multiplicity $2$. Meanwhile, accompanying eigenfunction estimates provide control over the number of nodal domains produced by the corresponding eigenfunctions as the neck width varies. In combination, these quantitative results further our understanding of nodal deficiency under small domain deformations. 

\subsection{Domain definition and statement of main results.}\label{sec:Domain} We study the nodal deficiency of Neumann eigenfunctions of a symmetric dumbbell domain denoted by $\Omega_{\eps}$. This dumbbell domain consists of two copies of the same piecewise smooth domain, joined by a thin neck that degenerates to a line segment as $\eps$ tends to 0. The Neumann spectrum of this domain contains pairs of eigenvalues degenerating to the same limit. The formal definition of $\Omega_{\eps}$ is then as follows:

\begin{define}[The Separated Domain, $\Omega_{0}$] \label{defn:dumbbell}
    Let $\Omega_L$ and $\Omega_R$ be a pair of disjoint, bounded, planar domains, with smooth boundaries except for a finite number of vertices. We define $\Omega_0 = \Omega_L\cup\Omega_R$ as the separated domain, symmetric across the vertical line $\{x=\tfrac{1}{2}\}$ and disjoint from the line segment $[0,1]\times\{0\}$. Moreover, for some fixed $\ell>0$ and setting $p_0=(0,0)$ and $p_1=(1,0)$, the boundary $\p\Omega_0$ is locally flat near $p_0,p_1$ in the sense that 
\begin{equation}
    \left\{(0,y): |y|<\ell\right\}\subset\p\Omega_L \quad \textrm{and} \quad \left\{(1,y): |y|<\ell\right\}\subset\p\Omega_R. \nonumber
\end{equation}
\end{define}

\begin{define}[The Dumbbell Domain, $\Omega_\eps$] Let $g\in C^1([0,1])$ be a positive function, satisfying $\norm{g}_{L^{\infty}} \leq 1$ with $g(x) = g(1-x)$ for all $x\in[0,1]$ and $g(0),g(1)<\ell$. For $0<\eps\leq 1$, let
\begin{align}\label{eq:NeckRe}
    R_{\eps}=\left\{(x,y):0<x<1:0<y<\eps g(x)\right\}
\end{align}
serve as a connecting neck, disjoint from the separated domain $\Omega_0$. We then define the dumbbell domain $\Omega_{\eps}$ as the perturbation of $\Omega_0$ given by
$$\Omega_\epsilon=\Omega_0\cup R_\epsilon\cup\left\{(0,y): 0<y<\epsilon g(0) \right\}\cup\left\{(1,y): 0<y<\eps g(1)\right\}.$$
\end{define}

\quad Under these definitions, $\Omega_\eps$ degenerates to the union of the separated domain $\Omega_0$ with the line segment $[0,1]\times\{0\}$ as $\epsilon\to 0$. See Figure \ref{fig:Dumbbell} for an example of $\Omega_\eps$ when $\Omega_L$ and $\Omega_R$ are square domains.
    
\begin{figure}[H]
    \centering
    \includegraphics[scale=0.7]{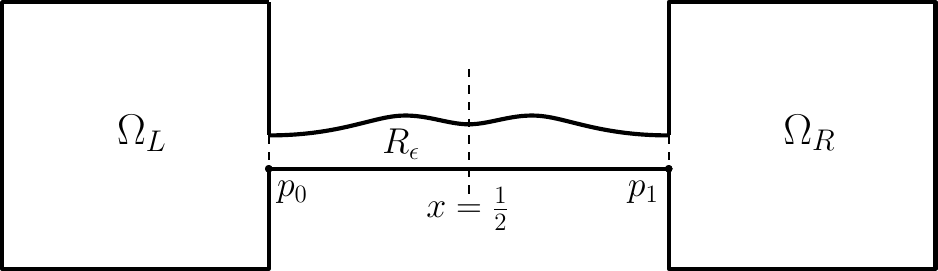}
    \caption{An example of the dumbbell domain $\Omega_\eps$.}
    \label{fig:Dumbbell}
\end{figure}

\quad It is known that eigenfunctions of the Dirichlet Laplacian, when defined over $\Omega_\eps$, become localized to the separated domain $\Omega_0$ as $\eps$ vanishes \cite{BCM-pleijel, D03, MS23}. The Dirichlet spectrum of $\Omega_\eps$ then converges to the Dirichlet spectrum of $\Omega_0$, and the degenerate neck plays no role. However, for small $\eps$, eigenfunctions of the Neumann Laplacian in $\Omega_\eps$ closely resemble the restrictions of eigenfunctions in $\Omega_0$ connected in $R_\eps$ by the solution to a Sturm-Liouville equation, as made precise in Theorem \ref{thm:EigFnVar} below. For example, the second Neumann eigenfunction of $\Omega_{\eps}$ resembles constants of different signs in $\Omega_0$ joined by a monotonic function of one variable in the neck. As shown in \cite{MS23} for general dumbbells, the nodal set of the second Neumann eigenfunction then cannot extend far from the neck. In general, let $\{\lambda_j^\epsilon\}_{j=1}^\infty$ denote the Neumann eigenvalues of $\Omega_{\eps}$ with corresponding $L^2(\Omega_\eps)$-normalized eigenfunctions $\{\varphi_j^\epsilon\}_{j=1}^\infty$, and set 
\begin{equation}\label{eq:EigenvalueList}
    \{\lambda_j\}_{j=1}^{\infty} = \{\mu_k\}_{k=1}^{\infty}\cup\{\tau_n\}_{n=1}^{\infty}
\end{equation}
as an ordered sequence, counted with multiplicity, where $\{\mu_k\}_{k=1}^{\infty}$ contains the Neumann eigenvalues of $\Omega_0$ and $\{\tau_n\}_{n=1}^{\infty}$ contains the Dirichlet eigenvalues of the Sturm-Liouville equation
\begin{align}\label{eq:OntheInt2}
    -(g(x)v'(x))' = \tau g(x)v(x) \text{ in } (0,1).
\end{align}
Then by \cite[Theorem $2.2$]{A95}, the limit $\lim_{\eps\to0}\lambda_j^\epsilon$ exists for each $j\in\mathbb{N}$ and is equal to $\lambda_j$. Thus, in contrast to the Dirichlet case, the geometry of $R_\eps$ impacts the Neumann spectrum of $\Omega_\eps$, which converges to the Neumann eigenvalues of $\Omega_0$ alongside the Dirichlet eigenvalues from (\ref{eq:OntheInt2}). 

\quad Note that by the symmetry of the separated domain $\Omega_0$, every Neumann eigenvalue of multiplicity $m$ for $\Omega_L$ appears in the sequence $\{\mu_k\}_{k=1}^{\infty}$ exactly $2m$ times. We consider the case where there are exactly two eigenvalue branches converging to an element $\mu\in\{\mu_k\}_{k=1}^{\infty}$ by making the following non-degeneracy assumption.
\begin{assumption}[Simple Eigenvalue Assumption]\label{ass:eig}
    We assume that $\mu\in \{\lambda_j\}_{j=1}^{\infty}$, where $\mu$ is a simple Neumann eigenvalue of $\Omega_L$, with $\mu\notin \{\tau_n\}_{n=1}^\infty$. 
\end{assumption}

For an eigenvalue $\mu$ satisfying this assumption, let $\phi$ be a corresponding Neumann eigenfunction in $\Omega_L$. Let $\varphi^{\eps,e}$ and $\varphi^{\eps,o}$ denote the pair of eigenfunctions in $\Omega_{\eps}$ with eigenvalues $\lambda^{\eps,e}$ and $\lambda^{\eps,o}$ converging to $\mu$. This notation is justified below in Lemma \ref{lem:SymmetryEigF}. Our first result provides a lower bound on the nodal deficiency of either eigenfunction $\varphi^{\eps,e}$ or $\varphi^{\eps,o}$ in $\Omega_{\eps}$ in terms of the nodal deficiency of $\phi$ in $\Omega_L$. This requires the following non-degeneracy assumption on the behavior of $\phi$ at $p_0\in\pa\Omega_L$, where we recall $p_0
$ is the point on $\pa\Omega_L$ where the neck $R_\eps$ is joined.

\begin{assumption}[Non-Vanishing Assumption]\label{ass:phi}
    We assume that $\phi(p_0)\neq0$, so that the nodal set of $\phi$ is disjoint from a neighborhood of $p_0\in\pa\Omega_L$. 
\end{assumption}

\quad Our first theorem considers the nodal deficiency (\ref{eq:NodDef}) and is as follows:

\begin{theorem} \label{thm:Main}
Under Assumptions \ref{ass:eig} and \ref{ass:phi}, there exists a constant $\eps_0>0$, depending only on the geometry of $\Omega_0$, the function $g$ defining $R_\eps$, and the eigenpair $(\mu,\phi)$, such that for $0<\eps<\eps_0$,
\begin{align*}
    \upvartheta(\varphi^{\eps,e},\Omega_\eps) \geq 2\upvartheta(\phi,\Omega_L) \quad \textrm{and} \quad
     \upvartheta(\varphi^{\eps,o},\Omega_\eps) \geq 2\upvartheta(\phi,\Omega_L),
\end{align*}
 with equality if the nodal set of $\phi$ in ${\Omega_L}$ contains no crossings.
\end{theorem}

\quad This result provides a lower bound for the nodal deficiency of the perturbed eigenfunctions $\varphi^{\eps,e}$ and $\varphi^{\eps,o}$ in $\Omega_\eps$ in terms of the nodal deficiency of $\phi$ in $\Omega_L$ and is proved in Sections \ref{sec:NDoriginal} and \ref{sec:NDperturbed}.  If the nodal set of $\phi$ has no crossings and hence does not self-intersect, then we find that, for small enough perturbations of the neck, the nodal sets of $\varphi^{\eps,e}$ and $\varphi^{\eps,o}$ likewise contain no crossings and $\llbracket \varphi^{\eps,e},\Omega_L\rrbracket=\llbracket \varphi^{\eps,o},\Omega_L\rrbracket=\llbracket \phi,\Omega_L\rrbracket$. As an immediate consequence of Theorem \ref{thm:Main}, we obtain a criterion for when a Courant sharp eigenfunction in $\Omega_L$ leads to a pair of Courant sharp eigenfunctions in the perturbed dumbbell domain $\Omega_\eps$.

\begin{cor} \label{cor:Main}
If $\phi$ is Courant sharp in $\Omega_L$ and its nodal set contains no crossings, then for all $0<\eps<\eps_0$, the eigenfunctions $\varphi^{\eps,e}$, $\varphi^{\eps,o}$ are a pair of Courant sharp eigenfunctions in $\Omega_{\eps}$.

Conversely, if $\phi$ is not Courant sharp in $\Omega_L$, then for all $0<\eps<\eps_0$, the eigenfunctions $\varphi^{\eps,e}$, $\varphi^{\eps,o}$ are not Courant sharp eigenfunctions in $\Omega_{\eps}$.
\end{cor}

\quad When $\mu=0$, the eigenfunction $\phi$ is constant in $\Omega_L$ and thus contains no crossings in its nodal set. In this case,  Corollary \ref{cor:Main} recovers the known fact that the first two eigenfunctions in $\Omega_\eps$ are necessarily Courant sharp.

\begin{rem} \label{rem:assum}
    If either Assumption \ref{ass:eig} or Assumption \ref{ass:phi} fails for a given separated domain $\Omega_0$ and neck function $g$, then under a generic perturbation of $\Omega_0$ and $g$, both assumptions hold.
\end{rem}

\subsubsection{Eigenfunction and eigenvalue estimates.} Key ingredients in the proof of Theorem \ref{thm:Main} include an approximation of the eigenfunctions $\varphi^{\eps,e}$ and $\varphi^{\eps,o}$ for small $\eps>0$ and a method for determining the position of the corresponding eigenvalues $\lambda^{\eps,e}$ and $\lambda^{\eps,o}$ in the Neumann spectrum of $\Omega_\eps$. 

\quad Given the eigenfunction $\phi$ solving (\ref{eq:PDE}) in $\Omega_L$ with eigenvalue $\mu$, suppose that $\phi^e$ and $\phi^o$ are even and odd reflections of $\phi$ across the vertical line $\left\{x=\tfrac{1}{2}\right\}$ to the separated domain $\Omega_0$. Noting that nodal sets are invariant under eigenfunction scaling, we normalize the eigenfunctions so that $\norm{\varphi^{\eps,e}}_{L^2(\Omega_{\eps})}=\norm{\varphi^{\eps,o}}_{L^2(\Omega_{\eps})}= 1 $, together with $\norm{\phi^e}_{L^2(\Omega_0)} =\norm{\phi^o}_{L^2(\Omega_0)} = 1$.  Using the same notation as above, we have that $(\lambda^{\eps,e},\varphi^{\eps,e})$ and $(\lambda^{\eps,o},\varphi^{\eps,o})$ are the pair of eigenpairs solving (\ref{eq:PDE}) in $\Omega_\eps$ and with $$\lim_{\eps\to0}\lambda^{\eps,e}=\lim_{\eps\to0}\lambda^{\eps,o}=\mu.$$ 

While $\mu$ is a simple Neumann eigenvalue of $\Omega_L$ by Assumption \ref{ass:eig}, the Sturm-Liouville eigenfunctions of \eqref{eq:OntheInt2} on the connecting interval $(0,1)$ still play a role in the spectral flow. We denote
\begin{equation}\label{eq:ThetaEq}
    \Theta_\lambda(a,b)=\int_0^1g(x)\left|\frac{d}{dx}\xi_\lambda^{a,b}(x)\right|^2dx-\lambda\int_0^1g(x)\left|\xi_\lambda^{a,b}(x)\right|^2dx
\end{equation}
where $\xi_\lambda^{a,b}$ solves (\ref{eq:OntheInt2}) with eigenvalue $\lambda$ and boundary conditions $\xi_\lambda^{a,b}(0)=a$ and $\xi_\lambda^{a,b}(1)=b$. The quantity in (\ref{eq:ThetaEq}) serves as a measure of how the energetic characterization of $\lambda$ differs from the symmetric Dirichlet form and offers the first-order variation for the eigenvalue branches given by $\lambda^{\eps,e}$ and $\lambda^{\eps,o}$, as described in the following statement.

\begin{theorem}\label{thm:EigVar}
    Under Assumption \ref{ass:eig}, the two eigenvalue branches on $\Omega_\epsilon$ that limit to $\mu$ as $\epsilon\to 0$ can be written as
    \begin{align}
        \lambda^{\epsilon,e}=\mu+\epsilon\Theta_{\mu}\left(\phi^e(p_0),\phi^e(p_1)\right)+o(\epsilon), \nonumber \\ \lambda^{\epsilon,o}=\mu+\epsilon\Theta_{\mu}\left(\phi^o(p_0),\phi^o(p_1)\right)+o(\epsilon).\nonumber
    \end{align}
\end{theorem}

\quad Here and throughout the paper, the implicit constant in $o(\eps)$ depends only on the geometry of $\Omega_0$, the function $g$ in (\ref{eq:NeckRe}), and the eigenpair $(\mu,\phi)$. Theorem \ref{thm:EigVar} provides the eigenvalue asymptotics needed to determine how the eigenvalue indices index$(\lambda^{\eps,e},\Omega_\eps)$ and index$(\lambda^{\eps,o},\Omega_\eps)$ change upon small perturbation of the domain.

\begin{rem} \label{rem:order}
    Given $\{\tau_n\}_{n=1}^\infty$ as in (\ref{eq:EigenvalueList}), let $k$ be the largest integer such that $\tau_k<\mu$, with $k=0$ if $0\leq \mu<\tau_1$. Then, under Assumption \ref{ass:eig} and by \cite[Theorem $2.2$]{A95}, there exists $\eps_0>0$ so that for $0<\eps<\eps_0$, $\lambda^{\eps,e}$ and $\lambda^{\eps,o}$ appear as the $2\:\emph{index}(\mu,\Omega_L) + k-1$ and $2\:\emph{index}(\mu,\Omega_L)+k$ Neumann eigenvalues of $\Omega_{\eps}$ in some order. Theorem \ref{thm:EigVar} says that this order is determined by the relative sizes of $\Theta_\mu\left(\phi^e(p_0),\phi^e(p_1)\right)$ and $\Theta_\mu\left(\phi^o(p_0),\phi^o(p_1)\right)$.
\end{rem}

\begin{rem}
    We will see that whenever $\mu$ is not a Neumann eigenvalue of (\ref{eq:OntheInt2}), both $\Theta_\mu\left(\phi^e(p_0),\phi^e(p_1)\right)$ and $\Theta_\mu\left(\phi^o(p_0),\phi^o(p_1)\right)$ are nonzero.  Moreover, under Assumptions \ref{ass:eig} and \ref{ass:phi}, it always holds that  $$\Theta_\mu\left(\phi^e(p_0),\phi^e(p_1)\right)\neq \Theta_\mu\left(\phi^o(p_0),\phi^o(p_1)\right),$$ and so in particular if either variation is zero, then the other is necessarily nonzero. For more details, see Lemma \ref{lem:ZeroCount} and Remark \ref{rem:BananaSplit}.
\end{rem}

\quad Theorem \ref{thm:EigVar} is an extension of Theorem 2.5(a) in \cite{A95}, which proved the same eigenvalue estimate in the case where $\mu$ is a simple Neumann eigenvalue of $\Omega_0$, disjoint from $\{\tau_n\}_{n=1}^{\infty}$, and hence there is only one eigenvalue branch of $\Omega_\eps$ converging to $\mu$. In this same setting, Theorem 2.7(a) of \cite{A95} provides estimates between the corresponding eigenfunctions of $\Omega_{\eps}$ and $\Omega_0$, also making use of an appropriate function $\xi^{a,b}_\mu$ in the neck. The analysis in Sections \ref{sec:Upper} and \ref{sec:Lower} extends these results to our setting. In the theorem below, we let $\xi^{\phi^e}_{\mu}$ be the function solving \eqref{eq:OntheInt2} with eigenvalue $\mu$ and boundary conditions $\xi^{\phi^e}_{\mu}(0) = \phi^e(p_0)$ and $\xi^{\phi^e}_{\mu}(1) = \phi^e(p_1)$. The function $\xi^{\phi^o}_{\mu}$ is defined analogously. 

\begin{theorem}\label{thm:EigFnVar}
    Under Assumption \ref{ass:eig}, the orthonormal eigenfunctions on $\Omega_\epsilon$ with eigenvalues $\lambda^{\epsilon,e},\lambda^{\epsilon,o}$ satisfy
    \begin{align}
        \norm{\varphi^{\epsilon,e}-\phi^e}_{H^1(\Omega_0)}^2=o(\epsilon), \quad \quad \norm{\varphi^{\epsilon,e}-\xi_{\mu}^{\phi^e}}_{H^1(R_\epsilon)}^2=o(\epsilon), \nonumber \\
        \norm{\varphi^{\epsilon,o}-\phi^o}_{H^1(\Omega_0)}^2=o(\epsilon), \quad \quad \norm{\varphi^{\epsilon,o}-\xi_{\mu}^{\phi^o}}_{H^1(R_\epsilon)}^2=o(\epsilon). \nonumber
    \end{align}
\end{theorem}

\quad This result implies that $\varphi^{\eps,e}$ and $\varphi^{\eps,o}$ closely resemble the respective eigenfunctions $\phi^e$ and $\phi^o$ in the separated domain $\Omega_0$, while in the neck, the perturbed eigenfunctions behave like extrusions of Sturm-Liouville solutions to (\ref{eq:OntheInt2}). This comparison is key to determining the number of nodal domains produced by the eigenfunctions for small positive $\epsilon$. Theorems \ref{thm:EigVar} and \ref{thm:EigFnVar} are proved in Sections \ref{sec:Upper} and \ref{sec:Lower} by extending the original proof of Arrieta from \cite{A95} to eigenvalues with multiplicity. The estimates in these theorems serve as the basis for studying the nodal deficiency (\ref{eq:NodDef}) in $\Omega_\eps$ and establishing Theorem \ref{thm:Main}.




\subsubsection{Spectral partitions.}\label{sec:CourantSharpMotiv} Motivation for the search for Courant sharp eigenfunctions was given by the work of Helffer, Hoffmann-Ostenhof, and Terracini upon discovery that nodal domains produced by a Courant sharp eigenfunction of (\ref{eq:PDE}) form a minimal, bipartite equipartition of $\Omega$ \cite{BH17}. Namely, the nodal domains of an eigenfunction of (\ref{eq:PDE}) minimize the energy functional
\begin{align*}
 \sup_{1\leq i \leq k}\left\{ \lambda_1(D_i)\,:\,\Omega = \bigcup_{j=1}^{k}D_j\right\}
\end{align*}
over all partitions $\Omega = \bigcup_{j=1}^{k}D_j$ into $k$ sub-domains \cite{HH13}. Here $\lambda_1(D_i)$ is the smallest Laplacian eigenvalue over $D_i$ with Dirichlet boundary conditions along $\p D_i\subset \Omega$ and Neumann boundary conditions along $\p D_i\cap \p\Omega$. Courant sharp eigenfunctions (with varying boundary conditions) have since been identified on squares, the torus, spheres, and other symmetric domains \cite{BBF17, HH14, HHT10, BS16}. The inequalities in Theorem \ref{thm:Main} extend these recent efforts in finding Courant sharp eigenfunctions.

\quad As an example, consider the case where $\Omega_L$ is the rectangle  $[0,M]\times[0,1]$, with $M^2>1$ irrational and close to $1$. Then, all Neumann eigenvalues of $\Omega_L$ are simple, and we can choose the neck-defining function $g$ and joining points $p_0$, $p_1$ so that Assumptions \ref{ass:eig} and \ref{ass:phi} hold for any $\mu$ in the Neumann spectrum of $\Omega_L$. Provided $M>1$ is sufficiently close to $1$, direct calculation yields that the $1$st, $2$nd, $4$th, and $9$th Neumann eigenfunctions are the only Courant sharp ones for $\Omega_L$. Figure \ref{fig:9CS} displays the nodal domains of each eigenfunction, labeled according to the corresponding eigenvalue's position in the spectrum, where color distinguishes sign.

\begin{figure}[H]
    \centering
    \includegraphics[scale=1]{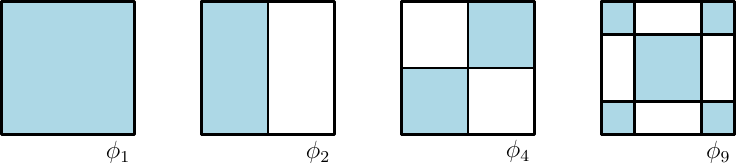}
    \caption{The Courant sharp eigenfunctions on an almost square domain.}
    \label{fig:9CS}
\end{figure}

\quad Note that the nodal sets of $\phi_1$, $\phi_2$ contain no  crossings in the closure of the domain, and hence if $\eps$ is small enough, then Corollary \ref{cor:Main} guarantees the existence of at least four Courant sharp eigenfunctions in $\Omega_\eps$. Since there are only four Courant sharp eigenfunctions for the rectangular domain $\Omega_L$, Corollary \ref{cor:Main} also indicates that for small $\eps>0$ there are no more than eight Courant sharp eigenfunctions of $\Omega_{\eps}$ corresponding to eigenvalue branches of $\Omega_{\eps}$ stemming from the eigenvalues in $\{\mu_k\}_{k=1}^\infty$, as given by (\ref{eq:EigenvalueList}). Note that the nodal sets of $\phi_4$ and $\phi_9$ contain self-intersections. Heuristically, we expect such crossings to open under domain perturbation, reducing the number of nodal domains accordingly. Therefore, we conjecture that, for generic length $M$ and neck-defining function $g$, there are exactly four Courant sharp eigenfunctions of $\Omega_{\eps}$ corresponding to eigenvalue branches of $\Omega_{\eps}$ with limits in $\{\mu_k\}_{k=1}^\infty$.

\quad We briefly illustrate Theorem \ref{thm:Main} and Corollary \ref{cor:Main} numerically using MATLAB. In Figures \ref{fig:CS} and \ref{fig:opening} below, MATLAB approximations to certain eigenfunctions of a symmetric rectangular dumbbell are shown. This dumbbell has a straight neck (that is, $g\equiv1$) with $\Omega_L = [0,M]\times[0,1]$, where $M^2 = (28)^{2/3}$ is irrational. 

\quad In Figure \ref{fig:CS}, the $6$th and $7$th Neumann eigenfunctions of the dumbbell are shown. This corresponds to a choice of eigenpair $(\mu,\phi)$ for the rectangle given by $\mu =\tfrac{4\pi^2}{M^2}$ and $\phi$ a multiple of $\cos(2\pi x/M)$. Since Assumptions \ref{ass:eig} and \ref{ass:phi} are satisfied and $\phi$ is the third Courant sharp eigenfunction of $\Omega_L$ with no crossings in its nodal set, Corollary \ref{cor:Main}, indicates that for sufficiently small neck width, we expect a Courant sharp pair of eigenfunctions in the rectangular dumbbell. The eigenvalue $\mu$ is between the first and second Dirichlet eigenvalues of the connecting interval $[0,2]$, and so this pair of eigenfunctions are the $6$th and $7$th Neumann eigenfunctions in the dumbbell. The nodal sets shown in Figure \ref{fig:CS} agree that these eigenfunctions are Courant sharp.

\begin{figure}[H]
    \centering
    \includegraphics[scale=0.3]{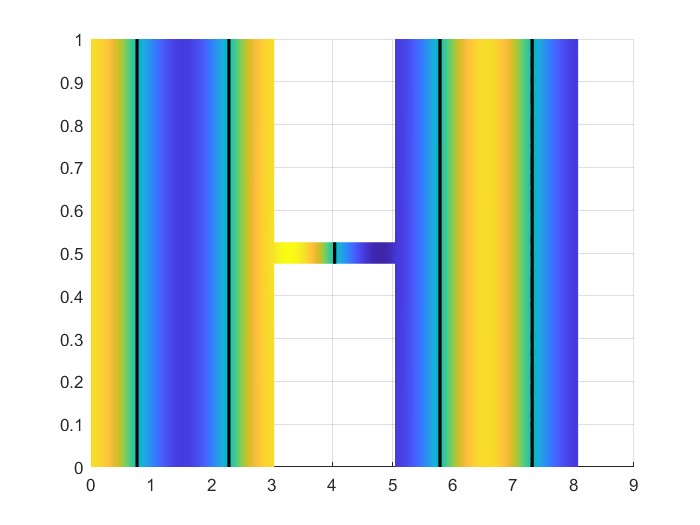}
    \includegraphics[scale=0.3]{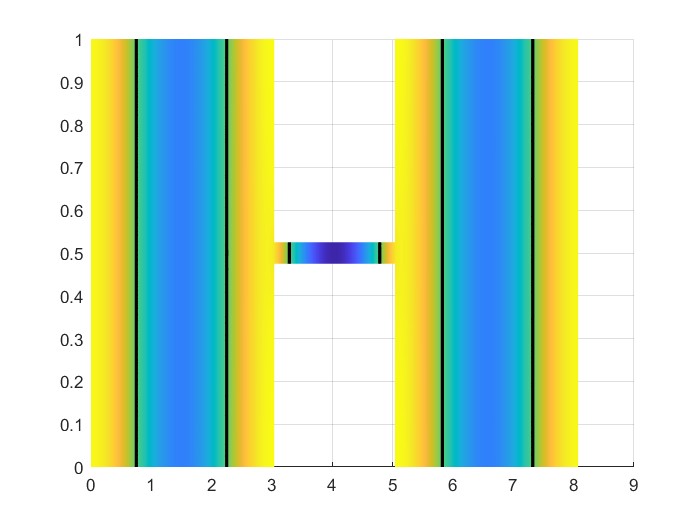}  
    \caption{A pair of Courant sharp eigenfunctions for a rectangular dumbbell.\\
    Figure credit: Steven Zucca}
    \label{fig:CS}
\end{figure}

\begin{figure}[H]
    \centering
    \includegraphics[scale=0.4]{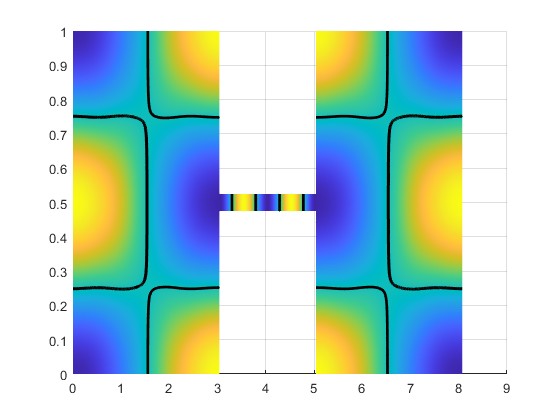}
   \caption{An eigenfunction of a rectangular dumbbell with no nodal crossings. \\
   Figure credit: Steven Zucca}
    \label{fig:opening}
\end{figure}

\quad In Figure \ref{fig:opening}, the $31$st Neumann eigenfunction of the dumbbell is shown. This corresponds to a choice of eigenpair $(\mu,\phi)$ for the rectangle given by $\mu = \tfrac{\pi^2}{M^2} + 4\pi^2$ and $\phi$ a multiple of $\cos(\pi x/M)\cos(2\pi y)$. Since Assumptions \ref{ass:eig} and \ref{ass:phi} are satisfied and $\phi$ has $6$ nodal domains, Theorem \ref{thm:Main} indicates that for sufficiently small neck width, the perturbed eigenfunction in the dumbbell must have nodal deficiency at least $2(14-6) = 16$. The eigenvalue $\mu$ is between the $4$th and $5$th Dirichlet eigenvalues of the connecting interval $[0,2]$, and so this eigenfunction is the $31$st Neumann eigenfunction in the dumbbell. The nodal set displayed in Figure \ref{fig:opening} indicates a nodal deficiency of $31-11 = 20 >16$, agreeing with Theorem \ref{thm:Main}. Note that this also gives evidence to our heuristic above: the crossings in the nodal set of $\phi$ have opened under this domain perturbation, leading to an increase in the nodal deficiency of the eigenfunction in the dumbbell.

\subsection{Outline.}\label{sec:Outline} The structure of the paper is as follows. Throughout this paper, we work under Assumption \ref{ass:eig}. In Sections \ref{sec:Upper} and \ref{sec:Lower} we prove the eigenvalue and eigenfunctions estimates from Theorems \ref{thm:EigVar} and \ref{thm:EigFnVar}, largely following the strategies featured in \cite[Sections $4-5$]{A95}. In Section \ref{sec:Upper}, we use the variational formulation of the eigenvalue to recover an upper bound for each eigenvalue branch stemming from $\mu$. The lower bound on the eigenvalues and estimates for the eigenfunction are then jointly established in Section \ref{sec:Lower}. In Section \ref{sec:NDoriginal}, we study the Sturm-Liouville equation (\ref{eq:OntheInt2}) under Assumption \ref{ass:phi} to relate the ordering of the eigenvalues in Theorem \ref{thm:EigVar} to the number of zeros of the odd and even solutions to this equation. Finally, in Section \ref{sec:NDperturbed}, we prove Theorem \ref{thm:Main} by counting the number of nodal domains of the eigenfunctions $\varphi^{\eps,e}$ and $\varphi^{\eps,o}$ in $R_\eps$ and by estimating from above the number of their nodal domains in $\Omega_L$ and $\Omega_R$. This final argument relies on the Faber-Krahn Theorem to rule out any nodal domains of sufficiently small area.

\section{An Eigenvalue Upper Bound}\label{sec:Upper}

\quad In this section, we establish, under Assumption \ref{ass:eig}, an upper bound for the eigenvalue branches $\lambda^{\eps,e},\lambda^{\eps,o}$, corresponding to orthonormal eigenfunctions $\varphi^{\eps,e},\varphi^{\eps,o}$, stemming from $\mu$. The analysis in this section follows the method of \cite[Section $4$]{A95}, which considered the case that $\mu$ was a simple Neumann eigenvalue of $\Omega_0 = \Omega_L\cup\Omega_R$. It pairs with an analogous eigenvalue estimate in Section \ref{sec:Lower} to establish Theorem \ref{thm:EigVar}. The eigenvalue upper bound is derived using the variational formulation of eigenvalues, so we first introduce some of the test functions that we use.

\quad Let $\{\mu_k\}_{k=1}^\infty$ denote the Neumann eigenvalues of (\ref{eq:PDE}) in $\Omega_0$ and $\{\tau_m\}_{m=1}^\infty$ the Dirichlet eigenvalues of (\ref{eq:OntheInt2}) on the unit interval. As before, we set $$\{\lambda_j\}_{j=1}^\infty=\{\mu_k\}_{k=1}^\infty\cup\{\tau_m\}_{m=1}^\infty, $$
ordered and counted with multiplicity. Using Assumption \ref{ass:eig}, suppose that  the eigenvalue $\mu$ of (\ref{eq:PDE}) in $\Omega_0$ corresponding to the eigenfunctions $\phi^e,\phi^o$ appears precisely as $\lambda_n$ and $\lambda_{n+1}$ in this sequence. Following the techniques featured in \cite{A95}, let $E:H^1(\Omega_0)\to H^1(\mathbb{R}^2)$ be a bounded, linear extension and define 
\begin{equation}
    \psi_j^\eps=\begin{cases}
        E\phi_k, & \textrm{if } \lambda_j=\mu_k\\ \begin{cases}0, & \textrm{in }\Omega_0 \\ \epsilon^{-1/2} \gamma_m, & \textrm{in }R_\eps\end{cases} & \textrm{if } \lambda_j=\tau_m
    \end{cases} \nonumber
\end{equation}
where $\phi_k$ is an $L^2(\Omega_0)$-normalized solution of (\ref{eq:PDE}) in $\Omega_0$ with eigenvalue $\mu_k$ and $\gamma_m$ is an $L^2([0,1])$-normalized solution of (\ref{eq:OntheInt2}) in $[0,1]$ with Dirichlet eigenvalue $\tau_m$. Since the eigenvalue $\mu_k$ appears an even number of times in the Neumann spectrum of symmetric $\Omega_0=\Omega_L\cup\Omega_R$, we may choose the $\phi_k$ so that they alternate between even and odd functions with respect to the vertical line $\{x=\tfrac{1}{2}\}$. The functions $\gamma_m$ are even and odd with respect to $\{x=\tfrac{1}{2}\}$ for $m$ odd and even respectively. We also choose the extension operator $E$ to preserve such symmetry. 

\quad Let $M(\lambda_n)$ be the multiplicity of $\lambda_n$ in $\{\mu_k\}_{k=1}^\infty\cup\{\tau_m\}_{m=1}^\infty$, and let $M_e(\lambda_n)$ and $M_o(\lambda_n)$ denote the number of even and odd functions $\psi_j^\eps$ among those $j$ with $\lambda_j=\lambda_n$. Note that $M_e(\lambda_n)+M_o(\lambda_n) = M(\lambda_n)$. Moreover, from \cite[Theorem $2.2$]{A95}, there are $M(\lambda_n)$ eigenvalue branches $\lambda_j^\eps$ of $\Omega_{\eps}$ with $\lim_{\eps\to0}\lambda_j^{\eps} = \lambda_n$. The following lemma establishes a correspondence between the even and odd functions $\psi_j^\eps$ and even and odd eigenfunctions of (\ref{eq:PDE}) in $\Omega_\eps$.

\begin{lemma}\label{lem:SymmetryEigF}
    For small $\epsilon>0$, there exists an orthonormal basis of eigenfunctions in $\Omega_{\eps}$ such that for each $\lambda_n\leq \mu$, among those eigenfunctions with eigenvalues $\lambda_j^{\eps}$ satisfying $\lim_{\eps\to0}\lambda_j^{\eps} = \lambda_n$, there are precisely $M_{e}(\lambda_n)$ and $M_o(\lambda_n)$ functions in this basis that are even and odd with respect to the vertical line $\left\{x=\tfrac{1}{2}\right\}$ respectively.
\end{lemma}

\quad In particular, applying this lemma for $\lambda_n=\mu$, with $M_e(\mu) = M_o(\mu)=1$ by Assumption \ref{ass:eig}, justifies the $e,o$ notation. Lemma \ref{lem:SymmetryEigF} states that for small perturbations, the symmetry of the domain $\Omega_0$ carries through to symmetry of the eigenfunctions in $\Omega_\eps$. Recall that by construction, the neck $R_\eps$ is also symmetric. Of particular importance is the fact that this lemma ensures that orthogonality is preserved in the sense that
\begin{equation}
    \left(\varphi^{\eps,e},\phi^o\right)_{\Omega_0}=\left(\varphi^{\eps,o},\phi^e\right)_{\Omega_0}=0 \nonumber
\end{equation}
for small $\epsilon>0$.

\begin{proof}[Proof of Lemma \ref{lem:SymmetryEigF}.] For $\psi\in H^1(\Omega_\eps)$, let $S\psi$ denote the reflection of $\psi$ across the vertical line $x=\tfrac{1}{2}$. If $\varphi$ solves (\ref{eq:PDE}) in $\Omega_\eps$ with eigenvalue $\lambda$, then $\varphi\pm S\varphi$ is either identically zero or another non-identically zero solution to (\ref{eq:PDE}) in $\Omega_\eps$ with eigenvalue $\lambda$. Note that $\varphi$ is even or odd across $x=\tfrac{1}{2}$ if and only if one of $\varphi\pm S\varphi$ is identically zero. Since $\Omega_{\eps}$ is symmetric with respect to the line $\{x=\tfrac{1}{2}\}$, we can simultaneously find an eigenbasis of $S$ and Neumann eigenfunctions of $\Omega_{\eps}$. In other words, we can find an orthonormal basis of Neumann eigenfunctions of $\Omega_{\eps}$ that are all odd or even with respect to $\{x=\tfrac{1}{2}\}$.

\quad To count the number of these eigenfunctions that, for small $\eps>0$, are even or odd, we use the second part of \cite[Theorem $2.2$]{A95}. This states that any eigenfunction $\varphi^{\eps}$ of (\ref{eq:PDE}) in $\Omega_{\eps}$ with eigenvalue converging to $\lambda_n$ must converge in $H^1(\Omega_{\eps})$ to its projection onto those $\psi_j^{\eps}$ with $\lambda_j=\lambda_n$. There are $M(\lambda_n)$ such eigenfunctions of $\Omega_{\eps}$ in the orthonormal basis and $M(\lambda_n)$ such $\psi_j^{\eps}$, and any even or odd eigenfunction $\varphi^{\eps}$ is automatically orthogonal to any odd or even $\psi_j^\eps$ respectively. Therefore, for small $\eps>0$, there must be the same number of even and odd eigenfunctions $\varphi^{\eps}$ as there are for the $\psi_j^{\eps}$, concluding the proof. 
\end{proof}

\quad To understand the spectral flow upon domain perturbation, we make use of the variational characterization of eigenvalues and consider the Rayleigh quotient with well-chosen test functions. In the following statement, we distinguish the eigenvalue branches stemming from $\mu$ according to order, supposing $\lambda^{\eps,\uparrow}\geq\lambda^{\eps,\downarrow}$.

\begin{lemma} \label{lem:LwrBd} Let $\lambda^{\eps,\downarrow},\lambda^{\eps,\uparrow}$ denote the upper and lower eigenvalue branches that approach $\mu$ as $\epsilon\to 0$. Then, for small $\epsilon>0$,
    \begin{align}
    \lambda^{\epsilon,\downarrow}&\leq \mu+\epsilon \min\left\{\Theta_{\mu}\left(\phi^e(p_0),\phi^e(p_1)\right), \Theta_{\mu}\left(\phi^o(p_0),\phi^o(p_1)\right)\right\}+o(\epsilon), \nonumber \\ 
    \lambda^{\epsilon,\uparrow}&\leq \mu+\epsilon \max\left\{\Theta_{\mu}\left(\phi^e(p_0),\phi^e(p_1)\right), \Theta_{\mu}\left(\phi^o(p_0),\phi^o(p_1)\right)\right\}+o(\epsilon), \nonumber
\end{align}
where $\Theta_\mu$ is given in (\ref{eq:ThetaEq}). 
\end{lemma}

\quad This result highlights the role of $\Theta_\mu$ in governing the first-order variation of the eigenvalues. It also shows that $\lambda^\epsilon-\mu\leq C\eps$ for either eigenvalue $\lambda^\eps\in\{\lambda^{\eps,e},\lambda^{\eps,o}\}$. In Section \ref{sec:Lower}, we pair Lemma \ref{lem:LwrBd} with a lower eigenvalue bound that identifies the branch order with respect to $\Theta_\mu$.

\begin{proof}[Proof of Lemma \ref{lem:LwrBd}.] Throughout this proof, we take the inner product to be in $L^2(\Omega_\eps)$, inducing the norm $\norm{\cdot}$ unless stated otherwise. For other sets $U$, we use $(\cdot,\cdot)_U$ and $\norm{\cdot}_U$ to denote the $L^2(U)$ inner product and norm.  For $f\in H^1(\Omega_0)$, we define
\begin{equation}\label{eq:Eextend}
    E_\eps^\mu f=\begin{cases}
        f(x,y), & \textrm{if } (x,y)\in\Omega_0 \\
        f(-x,y)+\frac{x}{\eps}\left(\xi^f_\mu(\eps)-f(-x,y)\right), & \textrm{if } (x,y)\in R_\eps, \: 0\leq x<\eps \\
        f(2-x,y)+\frac{1-x}{\eps}\left(\xi_\mu^f(1-\eps)-f(2-x,y)\right), & \textrm{if } (x,y)\in R_\eps, \: 1-\eps<x\leq 1 \\
        \xi_\mu^f(x), & \textrm{if } (x,y)\in R_\eps, \: \eps\leq x\leq1-\eps
    \end{cases}
\end{equation}
as an extension of $f$ to $\Omega_\eps$ such that $E_\eps^\mu f\in H^1(\Omega_\eps)$. Here $\xi_\mu^f$ denotes the solution to (\ref{eq:OntheInt2}) with eigenvalue $\mu$ and boundary conditions $\xi_\mu^f(0)=f(p_0)$ and $\xi_\mu^f(1)=f(p_1)$. We also set
\begin{equation}\label{eq:Etildeetend}
    \widetilde{E}_\eps^\mu f=\begin{cases}
        f(x,y), & \textrm{if } (x,y)\in\Omega_0 \\ \xi_\mu^f(x), & \textrm{if } (x,y)\in R_\eps
    \end{cases}
\end{equation}
as an extension of $f$ to $\Omega_\eps$ such that $\widetilde{E}_\eps^\mu f\in H^1(\Omega_0)\cap H^1(R_\eps)$ (but not necessarily in $H^1(\Omega_\eps)$). Note that ${E}_\eps^\mu f$ and $\widetilde{E}_\eps^\mu f$ both preserve any even or odd symmetry across $\{x=\tfrac{1}{2}\}$. For $\phi \in \{\phi^e,\phi^o\}$, the region over which the difference $E^\mu_\eps \phi-\widetilde{E}_\eps^\mu \phi$ is nonzero is in $R_\eps$ and of size $O(\epsilon^2)$. Moreover, since $\pa\Omega_L$ and $\pa\Omega_R$ contain a line segment centered at $p_0$ and $p_1$ respectively, $\phi$ and its gradient are bounded in a neighborhood of $p_0$ and $p_1$. Therefore, the operators in (\ref{eq:Eextend}) and (\ref{eq:Etildeetend}) closely approximate each other, in the sense that
\begin{equation}\label{eq:ExtendEst}
    \norm{E_\eps^\mu \phi-\widetilde{E}_\eps^\mu \phi}_{R_\eps}^2+\norm{\nabla\left(E_\eps^\mu \phi-\widetilde{E}_\eps^\mu \phi\right)}^2_{R_\eps}=O(\epsilon^2).
\end{equation}

\quad Suppose that there are $n-1$ entries in $\{\lambda_j\}_{j=1}^{\infty}$ with $\lambda_j<\mu$, so that by using \cite[Theorem 2.2]{A95}, $\lambda^{\eps,\downarrow}$ is the $n$th eigenvalue of $\Omega_{\eps}$ for small $\eps>0$. Set $\chi_\epsilon^e=E_\epsilon^{\mu}\phi^e$ and $\chi_\epsilon^o=E_\epsilon^{\mu}\phi^o$ as extensions of the even and odd eigenfunctions on the separated domain $\Omega_0$. There exist nonzero functions $$f^e\in W^e\eqdef\textrm{span}\{\psi_1^\epsilon,\dots,\psi_{n-1}^\epsilon, \chi_\epsilon^e\}$$ $$f^o\in W^o\eqdef\textrm{span}\{\psi_1^\epsilon, \dots,\psi_{n-1}^\epsilon,\chi_\epsilon^o\}$$ such that both $f^e,f^o$ are orthogonal to the first $n-1$ eigenfunctions $\varphi_1^\epsilon,\dots,\varphi_{n-1}^\epsilon$ solving (\ref{eq:PDE}) in $\Omega_\eps$. These functions exist because $\dim W^e=\dim W^o=n$. Then
\begin{equation}\label{eq:eitheror}
    \lambda^{\epsilon,\downarrow}\leq \min\left\{\frac{\norm{\nabla f^e}^2}{\norm{f^e}^2}, \frac{\norm{\nabla f^o}^2}{\norm{f^o}^2}\right\}. 
\end{equation}
because both $f^e,f^o$ serve as test functions for the Rayleigh quotient.

\quad Let $f\in\{f^e,f^o\}$ be the function that minimizes (\ref{eq:eitheror}) and let $\chi_\epsilon\in\{\chi_\epsilon^e,\chi_\epsilon^o\}$, $\phi\in\{\phi^e,\phi^o\}$ share the index of $f$. We can write $$f=\sum_{i=1}^{n-1}a_i\psi_i^\epsilon+a_n\chi_\epsilon$$ so that
\begin{equation}\label{eq:nonneg}
    \norm{\nabla f}^2-\lambda^{\epsilon,\downarrow}\norm{f}^2=\textbf{a}^T\textbf{Q}_\eps\textbf{a}\geq 0
\end{equation}
for $\textbf{a}^T=(a_1,\dots, a_n)$ and $\textbf{Q}_\eps\in\mathbb{R}^{n\times n}$ given by $(\textbf{Q}_\eps)_{ij}=q_{i,j}$ with
\begin{align}
    q_{i,j}&=q_{j,i}=\left(\nabla \psi_i^\epsilon,\nabla \psi_j^\epsilon\right)-\lambda^{\epsilon,\downarrow}\left(\psi_i^\epsilon,\psi_j^\epsilon\right) \quad \textrm{for} \quad 1\leq i,j\leq n-1 \nonumber \\
    q_{i,n}&=q_{n,i}=\left(\nabla \psi_i^\epsilon,\nabla \chi_\epsilon\right)-\lambda^{\epsilon,\downarrow}\left(\psi_i^\epsilon,\chi_\epsilon\right) \quad \textrm{for} \quad 1\leq i\leq n-1  \nonumber \\
    q_{n,n}&=\norm{\nabla\chi_\epsilon}^2-\lambda^{\epsilon,\downarrow}\norm{\chi_\epsilon}^2. \nonumber
\end{align}
Using \cite[Lemma $3.9$]{A95}, we have
\begin{align}
    \left(\nabla \psi_i^\epsilon,\nabla\psi_j^\epsilon\right)=O(\epsilon^{1/2}) \quad \textrm{and} \quad \left(\psi_i^\epsilon,\psi_j^\epsilon\right)=O(\epsilon^{1/2}) \nonumber
\end{align}
for $i\neq j$. Further, because $R_\epsilon$ is $O(\epsilon)$ in size,
\begin{equation}
    ||\nabla \psi_i^\epsilon||^2=\lambda_i+O(\epsilon) \quad \textrm{and} \quad ||\psi_i^\epsilon||^2=1+O(\epsilon). \nonumber
\end{equation}

\quad Using these results, we find that the diagonal elements of $\textbf{Q}_\eps$ approximate the spectral gaps
\begin{align}\label{eq:theproblem}
    q_{i,i}=\lambda_i-\lambda^{\epsilon,\downarrow}+O(\epsilon) 
\end{align}
for $i=1,\dots, n-1$. For small enough $\eps$, uniformly in $i$, these quantities are negative, since $\lim_{\eps\to0}\lambda^{\eps,\downarrow}=\mu=\lambda_n>\lambda_i$ for $1\leq i \leq n-1$. Meanwhile,
\begin{equation}
    q_{i,j}=O(\epsilon^{1/2}) \quad \textrm{for} \quad i\neq j \nonumber
\end{equation}
with $1\leq i,j\leq n-1$. Finally, for $1\leq i\leq n-1$, we recall $\chi_\epsilon=E^{\mu}_\epsilon\phi$ and so 
\begin{align}
    q_{i,n}&=\left(\nabla \psi_i^\epsilon,\nabla \chi_\epsilon\right)-\lambda^{\epsilon,\downarrow}\left(\psi_i^\epsilon,\chi_\epsilon\right)\nonumber \\
    &=\left(\nabla \psi_i^\epsilon, \nabla \widetilde{E}_\epsilon^{\mu}\phi\right)+\left(\nabla \psi_i^\epsilon, \nabla \left(E_\epsilon^{\mu}-\widetilde{E}_\epsilon^{\mu}\right)\phi\right)_{R_\eps}-\lambda^{\epsilon,\downarrow}\left(\psi_i^\epsilon, \widetilde{E}_\epsilon^{\mu}\phi\right)-\lambda^{\epsilon,\downarrow}\left(\psi_i^\epsilon, \left(E_\epsilon^{\mu}-\widetilde{E}_\epsilon^{\mu}\right)\phi\right)_{R_\eps} \nonumber \\
    &=\left(\nabla \psi_i^\epsilon, \nabla \widetilde{E}_\epsilon^{\mu}\phi\right)-\lambda^{\epsilon,\downarrow}\left(\psi_i^\epsilon, \widetilde{E}_\epsilon^{\mu}\phi\right)+O(\epsilon) \nonumber
\end{align}
using the estimates in (\ref{eq:ExtendEst}) and that $E_\eps^\mu\phi=\widetilde{E}_\eps^\mu \phi$ in $\Omega_0$. We can split the inner product over $\Omega_\epsilon$ into integrals over $\Omega_0$ and $R_\epsilon$, observing that the integrals over $\Omega_0$ are identically zero by orthogonality of eigenfunctions on the separated domain. Thus,
\begin{align}
    q_{i,n}&=\left(\nabla \psi_i^\epsilon, \nabla \xi_{\mu}^{\phi}\right)_{R_\epsilon}-\lambda^{\epsilon,\downarrow}\left(\psi_i^\epsilon, \xi_{\mu}^{\phi}\right)_{R_\epsilon}+O(\epsilon)=(\mu-\lambda^{\eps,\downarrow})\left( \psi_i^\eps,\xi_\mu^\phi\right)_{R_\eps}+O(\eps)=o(\epsilon^{1/2}) \nonumber 
\end{align}
by using integration by parts. The final equality above relies on the fact that $\lim_{\eps\to0}\lambda^{\eps,\downarrow}=\mu$.

\quad By (\ref{eq:nonneg}), there exists an eigenvalue $\textbf{q}_\eps\geq 0$ of $\textbf{Q}_\eps$, i.e. $$\det\left(\textbf{Q}_\eps-\textbf{q}_\eps I\right)=0.$$ We can expend the determinant of $\textbf{Q}_\eps-\textbf{q}_\eps I$ along the final column to get $q_{n,n}=\textbf{q}_\eps+o(\epsilon)$. Applying (\ref{eq:ExtendEst}), we find
\begin{align}
    q_{n,n}&=\norm{\nabla\chi_\epsilon}^2-\lambda^{\epsilon,\downarrow}\norm{\chi_\epsilon}^2=\norm{\nabla E_\epsilon^{\mu}\phi}^2-\lambda^{\epsilon,\downarrow}\norm{E_\epsilon^{\mu}\phi}^2\nonumber \\
    &=\left(\norm{\nabla \phi}_{\Omega_0}^2-\lambda^{\epsilon,\downarrow} \norm{\phi}^2_{\Omega_0}\right)+\left(\norm{\nabla\xi_{\mu}^{\phi}}_{R_\epsilon}^2-\lambda^{\epsilon,\downarrow}\norm{\xi_{\mu}^{\phi}}_{R_\epsilon}^2\right)+o(\epsilon) \nonumber
\end{align}
by again splitting the integral over $\Omega_\epsilon$ into integrals over $\Omega_0$ and $R_\epsilon$ and replacing $E_\eps^\mu\phi$ by $\widetilde{E}_\eps^\mu\phi$. Because $\xi_{\mu}^{\phi}$ is a function only of $x$, we find
\begin{align}
    q_{n,n}&=\mu-\lambda^{\epsilon,\downarrow}+\int_0^1\int_{0}^{\epsilon g(x)} \left(\left|\nabla\xi_{\mu}^{\phi}\right|^2-\lambda^\epsilon \left|\xi_{\mu}^{\phi}\right|^2\right) dydx\nonumber+o(\epsilon) \\
    &=\mu-\lambda^{\epsilon,\downarrow}+\epsilon\int_0^1 g(x)\left(\left|\frac{d}{dx}\xi_{\mu}^{\phi}\right|^2-\lambda^\epsilon \left|\xi_{\mu}^{\phi}\right|^2\right) dx\nonumber+o(\epsilon) \\
    &=\mu-\lambda^{\epsilon,\downarrow}+\epsilon\Theta_{\mu}\left(\phi(p_0),\phi(p_1)\right)+o(\epsilon) \nonumber
\end{align}
where $\Theta_\mu$ is given in (\ref{eq:ThetaEq}). Because $\textbf{q}_\eps\geq 0$ and $q_{n,n}=\textbf{q}_\eps+o(\epsilon)$, this implies the following upper bound on the lower branch eigenvalue:
\begin{equation}
    \lambda^{\epsilon,\downarrow}\leq \mu+\epsilon \min\left\{\Theta_{\mu}\left(\phi^e(p_0),\phi^e(p_1)\right), \Theta_{\mu}\left(\phi^o(p_0),\phi^o(p_1)\right)\right\}+o(\epsilon). \nonumber
\end{equation}

\quad We now determine an upper bound for the upper branch eigenvalue $\lambda^{\epsilon,\uparrow}$. Suppose first that $\varphi_n^{\eps}$ is even. Then, let $\psi_{i_1}^{\eps},\ldots,\psi_{i_m}^{\eps}$ with $i_m\leq n-1$ be those $\psi_j^{\eps}$ which are odd with respect to $\{x=\tfrac{1}{2}\}$. By Lemma \ref{lem:SymmetryEigF}, for small $\eps>0$, there are precisely $m$ eigenfunctions of $\Omega_{\eps}$ which are odd with respect to $\{x=\tfrac{1}{2}\}$, with eigenvalues satisfying $\lim_{\eps\to0}\lambda_j^{\eps} < \mu$. We label these eigenfunctions of $\Omega_{\eps}$ as $\varphi_{i_1}^\eps,\dots, \varphi_{i_m}^\eps$. Then, there exists a nonzero function
\begin{equation}
    f\in W'\eqdef\textrm{span}\{\psi_{i_1}^\epsilon,\dots,\psi_{i_m}^\epsilon, \chi^o_\epsilon\} \nonumber
\end{equation}
such that $f$ is orthogonal to these eigenfunctions $\varphi_{i_1}^\eps,\dots, \varphi_{i_m}^\eps$. This function exists because $\dim W'= m+1$. Using the symmetry of the domain, $f$ is also orthogonal to all of the even eigenfunctions solving (\ref{eq:PDE}) in $\Omega_\eps$ and hence $f$ is orthogonal to all of the first $n$ eigenfunctions $\varphi_1^\eps,\dots, \varphi_{n}^\eps$ in $\Omega_{\eps}$. In the same manner as before, the Rayleigh quotient of $f$ provides an upper bound for $\lambda^{\eps,\uparrow}$. 

\quad We construct a matrix $\textbf{Q}_\eps$ of size $(m+1)\times(m+1)$ so that
\begin{align*}
    f=\sum_{j=1}^{m}a_j\psi_{i_j}^{\eps} + a_{m+1}\chi^o_\eps, \qquad \norm{\nabla f}^2-\lambda^{\eps,\uparrow}\norm{f}^2=\textbf{a}^T\textbf{Q}_\eps\textbf{a}\geq0
\end{align*}
where $\textbf{a}^T=(a_1,\dots, a_{m+1})$. By constructing $f$ in this way, we avoid featuring the spectral gap $\lambda^{\eps,\uparrow}-\lambda^{\eps,\downarrow}$ from (\ref{eq:theproblem}) in $\textbf{Q}_\eps$. Rather, the matrix $\textbf{Q}_\eps$ has the same bounds as above on all diagonal and off-diagonal terms (since in particular $\lim_{\eps\to0}\lambda^{\eps,\uparrow}=\mu>\lambda_{i_j}$ for $1\leq j \leq m-1$), except now we obtain
\begin{align*}
   -o(\eps) \leq \ q_{m+1,m+1}= \norm{\nabla \chi^o_{\eps}}^2 - \lambda^{\eps,\uparrow}\norm{\chi^o_\eps}^2 = \mu- \lambda^{\eps,\uparrow} + \eps\Theta_\mu(\phi^o(p_0),\phi^o(p_1)) + o(\eps),
\end{align*}
giving
\begin{align*}
   \lambda^{\eps,\uparrow} \leq \mu+ \eps\Theta_\mu(\phi^o(p_0),\phi^o(p_1)) + o(\eps).
\end{align*}
\quad If instead $\varphi^\eps_n$ is odd, then we obtain an analogous result, replacing $\phi^o$ with $\phi^e$, and so we must have
\begin{equation}
    \lambda^{\epsilon,\uparrow}\leq \mu+\epsilon \max\left\{\Theta_{\mu}\left(\phi^e(p_0),\phi^e(p_1)\right), \Theta_{\mu}\left(\phi^o(p_0),\phi^o(p_1)\right)\right\}+o(\epsilon), \nonumber
\end{equation}
completing the proof.
\end{proof}

\section{An Eigenvalue Lower Bound}\label{sec:Lower}

\quad This section is dedicated to determining lower bounds for the eigenvalue branches $\lambda^{\eps,e}$ and $\lambda^{\eps,o}$ which, when combined with the results of Section \ref{sec:Upper}, provide an exact formula for the first-order eigenvalue variations in Theorem \ref{thm:EigVar}. These bounds are then used to obtain estimates on the corresponding eigenfunctions and prove Theorem \ref{thm:EigFnVar}. To accomplish this, we first present an estimate for the orthogonal projections of the perturbed eigenfunctions over the separated domain $\Omega_0=\Omega_L\cup\Omega_R$. The analysis in this section follows the method of \cite[Section $5$]{A95}, which again considered the case that $\mu$ was a simple Neumann eigenvalue of $\Omega_0 = \Omega_L\cup\Omega_R$.

\begin{prop}\label{prop:BlowItUp}
    Let $\updownarrow$ denote either $\uparrow$ or $\downarrow$. If we define
    \begin{equation}
        f_\epsilon^\updownarrow=\varphi^{\epsilon,\updownarrow}-\left(\varphi^{\epsilon,\updownarrow}, \phi^e\right)_{\Omega_0}\phi^e-\left(\varphi^{\epsilon,\updownarrow},\phi^o\right)_{\Omega_0}\phi^o, \nonumber
    \end{equation}
    then $$\norm{f_\epsilon^{\updownarrow}}_{\Omega_0}^2=o\left(\norm{\nabla f_\epsilon^{\updownarrow}}_{\Omega_0}^2\right).$$
\end{prop}

\quad Notice that the inner products in this proposition are over the separated domain $\Omega_0$, where the eigenfunctions $\phi^e,\phi^o$ are defined. The functions $f_\epsilon^\updownarrow$ are constructed to be orthogonal to the eigenfunctions $\phi^e,\phi^o$ with respect to the $L^2(\Omega_0)$ inner product. Thus, Proposition \ref{prop:BlowItUp} describes the asymptotic behavior of the Dirichlet form of the orthogonal projections of $\varphi^{\eps,e},\varphi^{\eps,o}$ over the eigenbasis $\{\phi^e,\phi^o\}$. The proof of Proposition \ref{prop:BlowItUp} follows directly from \cite[Proposition $5.1$]{A95}.

\quad Hereafter, let $\varphi^\epsilon\in\{\varphi^{\epsilon,e},\varphi^{\epsilon,o}\}$ solve (\ref{eq:PDE}) in $\Omega_\eps$ with eigenvalue $\lambda^\eps\in\{\lambda^{\eps,e},\lambda^{\eps,o}\}$, as the following statements in this section hold regardless of the label. We find that for small enough $\epsilon$, the perturbed eigenfunction $\varphi^{\eps}$ closely resembles the extrusion of a solution to (\ref{eq:OntheInt2}) over the neck $R_\epsilon$. Let $M:H^1(R_\epsilon)\to H^1((0,1))$ be an operator given by
\begin{equation}\label{eq:Avg}
    Mf(x)=\frac{1}{\epsilon g(x)}\int_0^{\epsilon g(x)} f(x,y)dy
\end{equation}
which averages a function in the $y$ direction over $R_\epsilon$. If we let $\xi_\epsilon$ solve the Sturm-Liouville equation (\ref{eq:OntheInt2}) over $(0,1)$ with eigenvalue $\lambda^\epsilon$ and boundary conditions 
\begin{equation}\label{eq:xissss}
        \xi_\epsilon(0)=T_L\varphi^\epsilon, \quad \textrm{and} \quad \xi_\epsilon(1)=T_R\varphi^\epsilon 
\end{equation}
where $T_L, T_R$ average a function at the left and right endpoints of the channel, i.e. $T_Lf=Mf(0)$ and $T_Rf=Mf(1)$ with $M$ as in (\ref{eq:Avg}), then $\xi_\eps$ closely approximates $\varphi^{\eps}$ in $R_\eps$, as made precise in the following statement.

\begin{prop}\label{prop:specgap}
    As in (\ref{eq:EigenvalueList}), let $\{\tau_n\}_{n=1}^\infty$ denote the Dirichlet eigenvalues of (\ref{eq:OntheInt2}). Under Assumption \ref{ass:eig}, there exists a positive constant $C$ such that
    \begin{equation}
        \norm{\varphi^\epsilon-\xi_\epsilon}_{R_\epsilon}^2\leq C\epsilon^2 \norm{\frac{\p\varphi^\epsilon}{\p y}}_{R_\epsilon}^2\left(1+\sum_{i=1}^\infty \frac{\tau_i}{\left(\lambda^\epsilon-\tau_i\right)^2}\right). \nonumber
    \end{equation}
\end{prop}

\quad Assumption \ref{ass:eig} that $\mu$ is not a Dirichlet eigenvalue of (\ref{eq:OntheInt2}) keeps the inequality in Proposition \ref{prop:specgap} finite. The estimates of \cite[Theorem $2.2$]{A95} guarantee that $\lambda^\eps-\mu=o(1)$ and hence $\lambda^\epsilon$ is bounded away from $\{\tau_n\}_{n=1}^\infty$ for sufficiently small $\epsilon$. The proof of Proposition \ref{prop:specgap} follows directly from that of \cite[Proposition $5.3$]{A95}. In combination, Propositions \ref{prop:BlowItUp} and \ref{prop:specgap} are sufficient to determine a lower bound for $\lambda^\epsilon$, following the same proof strategy as in case (a) of \cite[Section $5$]{A95}. In contrast to Lemma \ref{lem:LwrBd}, the following statement distinguishes between branches according to eigenfunction symmetry, as established in Lemma \ref{lem:SymmetryEigF}.

\begin{lemma}\label{lem:LowerBd}
    Let $\lambda^{\eps}\in\{\lambda^{\eps,e},\lambda^{\eps,o}\}$ denote either of the eigenvalue branches that approach $\mu$ as $\epsilon\to 0$. For small $\epsilon$,
    \begin{align}
        \lambda^{\eps}\geq \mu+\epsilon\Theta_\mu\left(\phi(p_0),\phi(p_1)\right)+o(\epsilon) \nonumber
    \end{align}
    where $\Theta_\mu$ is given in (\ref{eq:ThetaEq}) and $\phi\in\{\phi^e,\phi^o\}$ matches the index of $\lambda^\eps$.
\end{lemma}

\quad Alongside Lemma \ref{lem:LwrBd}, this result identifies $\Theta_\mu$ in (\ref{eq:ThetaEq}) as the first-order eigenvalue variation and proves the estimates in Theorem \ref{thm:EigVar}. The proof of Lemma \ref{lem:LowerBd} also sets the framework for establishing the eigenfunction estimates in Theorem \ref{thm:EigFnVar}.

\begin{proof}[Proof of Lemma \ref{lem:LowerBd}.] Under Assumption \ref{ass:eig}, \cite[Proposition $5.4$]{A95} states that there exist positive constant $C_1,C_2$ such that 
{\allowdisplaybreaks
\begin{align}\label{eq:startingLower}
    \begin{split}
    \lambda^\epsilon\geq  \norm{\nabla \varphi^\epsilon}_{\Omega_0}^2&+\lambda^\epsilon\norm{\varphi^\epsilon}_{R_\epsilon}^2+C_1\norm{\frac{\p\varphi^\epsilon}{\p y}}_{R_\epsilon}^2+\norm{\frac{\p\varphi^\epsilon}{\p x}-\xi_\epsilon'}_{R_\epsilon}^2  \\ &-\lambda^\epsilon\norm{\varphi^\epsilon-\xi_\epsilon}_{R_\epsilon}^2-C_2\epsilon^3\int_0^1g\left|\xi_\epsilon'\right|^2+\epsilon\Theta_{\lambda^\epsilon}\left(T_L\varphi^\epsilon, T_R\varphi^\epsilon\right). \end{split}
\end{align}
}

We use Proposition \ref{prop:specgap} and the fact that $\norm{\varphi^\epsilon}_{\Omega_\epsilon}^2=1$ to rewrite (\ref{eq:startingLower}) as
\begin{align}
    0\geq  \norm{\nabla \varphi^\epsilon}_{\Omega_0}^2-\lambda^\epsilon\norm{\varphi^\epsilon}_{\Omega_0}^2+C_1\norm{\frac{\p\varphi^\epsilon}{\p y}}_{R_\epsilon}^2+\norm{\frac{\p\varphi^\epsilon}{\p x}-\xi_\epsilon'}_{R_\epsilon}^2-C_2\epsilon^3\int_0^1g\left|\xi_\epsilon'\right|^2+\epsilon\Theta_{\lambda^\epsilon}\left(T_L\varphi^\epsilon, T_R\varphi^\epsilon\right). \nonumber
\end{align}
for some new positive constant $C_1$. For the sake of notation, this can be simplified to
\begin{align}\label{eq:GettingThere}
    0\geq & \norm{\nabla \varphi^\epsilon}_{\Omega_0}^2-\lambda^\epsilon\norm{\varphi^\epsilon}_{\Omega_0}^2-C_2\epsilon^3\int_0^1g\left|\xi_\epsilon'\right|^2+\epsilon\Theta_{\lambda^\epsilon}\left(T_L\varphi^\epsilon, T_R\varphi^\epsilon\right) +R\left(\varphi^\epsilon\right)
\end{align}
where $R(\varphi^\eps)$ denotes a nonnegative remainder term given by
\begin{equation}
    R\left(\varphi^\epsilon\right)\eqdef C_1\norm{\frac{\p\varphi^\epsilon}{\p y}}_{R_\epsilon}^2+\norm{\frac{\p\varphi^\epsilon}{\p x}-\xi_\epsilon'}_{R_\epsilon}^2\geq 0, \nonumber
\end{equation}
where we recall that $\xi_\epsilon$ solves (\ref{eq:OntheInt2}) with eigenvalue $\lambda^\eps$ and boundary conditions (\ref{eq:xissss}).

\quad Let $a=\left(\varphi^\epsilon,\phi^e\right)_{\Omega_0}$ and $b=\left(\varphi^\epsilon,\phi^o\right)_{\Omega_0}$ so that
\begin{align}
    \norm{\nabla \varphi^\epsilon}_{\Omega_0}^2&=\norm{\nabla\varphi^\epsilon-a\nabla\phi^e-b\nabla\phi^o}_{\Omega_0}^2+\mu\left(|a|^2+|b|^2\right), \nonumber \\
\norm{\varphi^\epsilon}_{\Omega_0}^2&=\norm{\varphi^\epsilon- a\phi^e -b\phi^o }_{\Omega_0}^2+\left(|a|^2+|b|^2\right). \nonumber
\end{align}

According to \cite[Theorem $2.2$]{A95}, the coefficients satisfy $|a|^2+|b|^2=1+o(1)$, although we improve this estimate further along in this proof. Proposition \ref{prop:BlowItUp} states that
$$\norm{\varphi^\epsilon-a\phi^e-b\phi^o}_{\Omega_0}^2=o\left(\norm{\nabla\varphi^\epsilon-a\nabla\phi^e-b\nabla\phi^o}_{\Omega_0}^2\right)$$
and so (\ref{eq:GettingThere}) implies 
{\allowdisplaybreaks
\begin{align}\label{eq:GettingThere2}
\begin{split}
    \left(\lambda^\epsilon-\mu\right)(1+o(1))\geq \norm{\nabla\varphi^\epsilon-a\nabla\phi^e-b\nabla\phi^o}_{\Omega_0}^2\left(1+o(1)\right)&-C_2\epsilon^3\int_0^1 g\left|\xi_\epsilon'\right|^2 \\ &+\epsilon\Theta_{\lambda^\epsilon}\left(T_L\varphi^\epsilon, T_R\varphi^\epsilon\right) +R\left(\varphi^\epsilon\right). 
\end{split}
\end{align}}

To collect the nonnegative terms, we set
\begin{equation}\label{eq:Reps}
    \widetilde{R}\left(\varphi^\epsilon\right)\eqdef\norm{\nabla\varphi^\epsilon-a\nabla\phi^e-b\nabla\phi^o}_{\Omega_0}^2\left(1+o(1)\right)+R\left(\varphi^\epsilon\right)\geq 0
\end{equation}
so that (\ref{eq:GettingThere2}) further reduces to
\begin{equation}\label{eq:GettingThere3}
    \left(\lambda^\epsilon-\mu\right)(1+o(1))\geq\epsilon\Theta_{\lambda^\epsilon}\left(T_L\varphi^\epsilon, T_R\varphi^\epsilon\right) -C_2\epsilon^3\int_0^1 g\left|\xi_\epsilon'\right|^2+\widetilde{R}\left(\varphi^\epsilon\right).
\end{equation}

\quad We claim that $T_L\varphi^\epsilon \to a\phi^e(p_0)+b\phi^o(p_0)$ and $T_R\varphi^\epsilon\to a\phi^e(p_1)+b\phi^o(p_1)$ as $\epsilon\to 0$. If this were not the case, then \cite[Lemma $3.6$]{A95} would imply that 
\begin{align*}
\epsilon\Theta_{\lambda^\epsilon}\left(T_L\varphi^\epsilon, T_R\varphi^\epsilon\right) -C_2\epsilon^3\int_0^1 g\left|\xi_\epsilon'\right|^2+\widetilde{R}\left(\varphi^\epsilon\right)
\end{align*}
 is bounded below by a multiple of $\left|\ln\epsilon\right|^{-1}$. Therefore, this would allow us to conclude from (\ref{eq:GettingThere3}) that $\lambda^\epsilon-\mu$ is similarly bounded below by a multiple of $\left|\ln\epsilon\right|^{-1}$. The upper eigenvalue bounds from Lemma \ref{lem:LowerBd} however imply that $\lambda^\epsilon-\mu$ is bounded by a multiple of $\eps$, giving us a contradiction. Therefore, we must have that 
\begin{align}
    T_L\varphi^\epsilon \to a\phi^e(p_0)+b\phi^o(p_0) \nonumber\\ T_R\varphi^\epsilon\to a\phi^e(p_1)+b\phi^o(p_1) \nonumber
\end{align}
as $\epsilon\to 0$. So far, we have not determined $a,b$ beyond the estimate $|a|^2+|b|^2=o(1)$, but the symmetry of $\Omega_\epsilon$ means that only one can be nonzero, as in Lemma \ref{lem:SymmetryEigF}. More precisely, taking $\varphi^\epsilon=\varphi^{\epsilon,e}$ sets $b=0$ while taking $\varphi^\epsilon=\varphi^{\epsilon,o}$ sets $a=0$. Separating the eigenvalue branches according to symmetry, equation (\ref{eq:GettingThere3}) therefore provides the lower bounds
{\allowdisplaybreaks
\begin{align}\label{eq:lowerbd}
    \begin{split}
    \lambda^{\epsilon,e}&\geq\mu+\epsilon\Theta_{\mu}\left(\phi^e(p_0),\phi^e(p_1)\right)+\widetilde{R}\left(\varphi^{\epsilon,e}\right)+o(\epsilon), \\ 
    \lambda^{\epsilon,o}&\geq\mu+\epsilon\Theta_{\mu}\left(\phi^o(p_0),\phi^o(p_1)\right)+\widetilde{R}\left(\varphi^{\epsilon,0}\right)+o(\epsilon). 
    \end{split}
\end{align}}

\quad Combined with the eigenvalue upper bounds from Section \ref{sec:Upper}, the inequalities in (\ref{eq:lowerbd}) determine the leading-order eigenvalue variations and finish the proof of Theorem \ref{thm:EigVar}. Further, they imply that the nonnegative remainders must satisfy
\begin{equation}\label{eq:smallR}
    \widetilde{R}\left(\varphi^{\epsilon,e}\right)=o(\epsilon) \quad \textrm{and} \quad \widetilde{R}\left(\varphi^{\epsilon,o}\right)=o(\epsilon),
\end{equation}
which we can use to approximate the eigenfunctions $\varphi^{\eps,e}$, $\varphi^{\eps,o}$ and prove Theorem \ref{thm:EigFnVar}, once again following \cite[Section $5$]{A95}.
\end{proof}

\quad The eigenfunction estimates over the neck $R_\epsilon$ in Theorem \ref{thm:EigFnVar} follow from Proposition \ref{prop:specgap}, (\ref{eq:smallR}), and \cite[Lemma $3.1$]{A95}, which shows that solutions of \eqref{eq:OntheInt2} are continuous with respect to their boundary values. Namely, because the area of $R_\epsilon$ is $O(\epsilon)$,
\begin{equation}
    \norm{\varphi^{\epsilon,e}-\xi_{\mu}^{\phi^e}}_{H^1(R_\epsilon)}^2\leq 2\norm{\varphi^{\epsilon,e}-\xi_\epsilon^e}_{H^1(R_\epsilon)}^2+2\norm{\xi_\epsilon^e-\xi_\mu^{\phi^e}}_{H^1(R_\epsilon)}^2=o(\epsilon) \nonumber
\end{equation}
where $\xi_\mu^{\phi^e}$ solves (\ref{eq:OntheInt2}) with eigenvalue $\mu$ and boundary conditions given by $\xi_\mu^{\phi^e}(0)=\phi^e(p_0)$ and $\xi_\mu^{\phi^e}(1)=\phi^e(p_1)$. Meanwhile, because all of the terms in (\ref{eq:Reps}) are individually nonnegative, (\ref{eq:smallR}) implies
\begin{equation}
    \norm{\nabla\left(\varphi^{\epsilon,e}-\left(\varphi^{\epsilon,e},\phi^e\right)_{\Omega_0}\phi^e\right)}_{\Omega_0}^2=o(\eps),
\end{equation}
and hence, alongside Proposition \ref{prop:BlowItUp},
\begin{equation}\label{eq:closeH1}
    \norm{\varphi^{\epsilon,e}-\left(\varphi^{\epsilon,e},\phi^e\right)_{\Omega_0}\phi^e}_{H^1(\Omega_0)}^2=o(\eps).
\end{equation}

\quad To conclude the proof of Theorem \ref{thm:EigFnVar}, we demonstrate how to use this to establish eigenfunction estimates over $\Omega_0$. Using the triangle inequality, (\ref{eq:closeH1}) provides the existence of some constant $C$, that may change line to line, such that 
\begin{align}\label{eq:wrapITup}
    \norm{\varphi^{\epsilon,e}-\phi^e}_{H^1(\Omega_0)}^2\leq C\left(1-\left(\varphi^{\epsilon,e},\phi^e\right)\right)^2+o(\epsilon).
\end{align}
Because $\norm{\varphi^{\epsilon,e}}_{\Omega_\epsilon}^2=1$ and $(\varphi^{\eps,e},\phi^e)=1+o(1)$, we have
\begin{align}
    1-\left(\varphi^{\epsilon,e},\phi^e\right)\leq C\left(1-\left(\varphi^{\epsilon,e},\phi^e\right)^2\right)&=C\left(\norm{\varphi^{\epsilon,e}-\left(\varphi^{\epsilon,e},\phi^e\right)\phi^e}_{\Omega_0}^2+\norm{\varphi^{\epsilon,e}}_{R_\epsilon}^2\right) \nonumber \\ &\leq C\left(\norm{\varphi^{\epsilon,e}-\xi_{\mu}^{\phi^e}}_{R_\epsilon}^2+\norm{\xi_{\mu}^{\phi^e}}_{R_\epsilon}^2\right)+o(\epsilon). \nonumber
\end{align}
The previously-established eigenfunction estimate over $R_\epsilon$ and the fact that the channel is of size $O(\epsilon)$ therefore imply that $\left(\varphi^{\epsilon,e},\phi^e\right)=1+O(\epsilon)$ and hence (\ref{eq:wrapITup}) provides the desired even eigenfunction estimate over $\Omega_0$. The same analysis holds for $\varphi^{\epsilon,o}$, completing the proof of Theorem \ref{thm:EigFnVar}.

\section{Eigenvalue Branch Order}\label{sec:NDoriginal}

\quad In this section, we determine the ordering of the eigenvalue branches $\lambda^{\eps,e},\lambda^{\eps,o}$ for small $\epsilon$ by establishing a connection between the first-order eigenvalue variation featured in Theorem \ref{thm:EigVar} and the number of zeros produced by the Sturm-Liouville eigenfunctions $\xi_\mu^{\phi^e}$, $\xi_\mu^{\phi^o}$ featured in Theorem \ref{thm:EigFnVar}. The number of eigenfunction zeros is dependent on the eigenvalue $\mu$ and can be quantified by a comparison of $\mu$ with the Dirichlet eigenvalues of (\ref{eq:OntheInt2}). These results are made precise in Proposition \ref{prop:index-main}. From now on we work under both Assumptions \ref{ass:eig} and \ref{ass:phi}, which in particular ensures that $\mu$ is not a Dirichlet eigenvalue of \eqref{eq:OntheInt2} and that the Sturm-Liouville eigenfunctions $\xi_\mu^{\phi^e}$, $\xi_\mu^{\phi^o}$  are not identically zero.

\quad For the sake of notation, set $\xi^e=\xi_\mu^{\phi^e}$ and $\xi^o=\xi_\mu^{\phi^o}$ as eigenfunctions solving the differential equation in (\ref{eq:OntheInt2}) with eigenvalue $\mu$. Note that $\xi^e$ is even and $\xi^o$ is odd, satisfying the boundary conditions
\begin{align}
    \xi^e(0)&=\xi^e(1)=\phi^e(p_0), \nonumber \\
    \xi^o(0)&=-\xi^o(1)=\phi^o(p_0). \nonumber
\end{align}
By symmetry, integration by parts, and the equation satisfied by $\xi^e$ and $\xi^o$, we can express the eigenvalue variation (\ref{eq:ThetaEq}) in terms of eigenfunction behavior at either endpoint of $(0,1)$. For consistency, we write
\begin{equation}\label{eq:Theta1and2}
    \Theta_{\mu}\left(\phi^e(p_0),\phi^e(p_1)\right)=-2g(0)\xi^e(0){\xi^{e}}'(0) \quad \textrm{and} \quad \Theta_{\mu}\left(\phi^o(p_0),\phi^o(p_1)\right)=-2g(0)\xi^o(0){\xi^{o}}'(0).
\end{equation}
Note that when $\mu=0$, the even eigenfunction $\xi^e$ is the constant function and $\Theta_0\left(\phi^e(p_0),\phi^e(p_1)\right)=0$. Meanwhile, we can express the odd eigenfunction $\xi^o$ as a multiple of
\begin{equation}\label{eq:oddxie}
    \int_{\frac{1}{2}}^x\frac{1}{g(t)}dt,
\end{equation}
which implies that $\Theta_0\left(\phi^o(p_0),\phi^o(p_1)\right)$ is positive and the spectral flow bifurcates upon this domain perturbation. For larger eigenvalues, we find that this splitting of eigenvalue branches remains.

\begin{rem}\label{rem:DifferingSigns}
    When $g\equiv 1$, the eigenfunctions $\xi^e,\xi^o$ are explicit for all $\mu$. Under our assumptions and according to the simplified expressions in (\ref{eq:Theta1and2}), we find that, for this particular boundary function, $\Theta_\mu\left(\phi^e(p_0),\phi^e(p_1)\right)$ and $\Theta_\mu\left(\phi^o(p_0),\phi^o(p_1)\right)$ differ in sign whenever $\mu$ is positive.
\end{rem}

\quad It is known \cite[Chapter $8$]{CL55} that the zeros of $\xi^e,\xi^o$ interlace, and by symmetry, the number of their zeros must differ by exactly one. Further, because $\xi^e$ and $\xi^o$ satisfy a non-degenerate Sturm-Liouville equation, their zeros are simple and ${\xi^e}'$, ${\xi^o}'$ are non-vanishing at the zeros. Therefore, since $\xi^e(0)$, $\xi^e(1)$ have the same sign, $\xi^e$ must have an even number of zeros, while analogously $\xi^o$ must have an odd number of zeros. We hereafter denote $N_e,N_o$ as the number of zeros of $\xi^e,\xi^o$ in $(0,1)$ respectively.

\begin{lemma}\label{lem:NumEvOd}
    Let $k$ be the largest integer such that $\tau_k<\mu$, with $k=0$ if $0\leq \mu<\tau_1$, where $\{\tau_n\}_{n=1}^\infty$ denotes the Dirichlet eigenvalues of (\ref{eq:OntheInt2}). Then
    \begin{equation}
        N_e=\begin{cases} k, & k \textrm{ even} \\ k+1, & k \textrm{ odd}\end{cases} \quad \quad \textrm{and} \quad \quad N_o=\begin{cases} k+1, & k \textrm{ even} \\ k, & k \textrm{ odd}.\end{cases} \nonumber
    \end{equation}.
\end{lemma}

\quad By comparing the eigenvalue $\mu$ with the Dirichlet spectrum, this result thus provides an exact count for the number of nodal domains produced by the eigenfunctions $\xi^e,\xi^o$ on the interval $(0,1)$. Recall that under Assumption \ref{ass:eig}, $\mu$ does not intersect $\{\tau_n\}_{n=1}^\infty$. As a special case, if $\mu=0$, then $\xi^e$ is the constant function with empty nodal set, while from (\ref{eq:oddxie}), the function $\xi^o$ has a zero at $x=\tfrac{1}{2}$. Because the boundary function $g$ is positive and bounded below, this is the only zero of $\xi^o$ in $(0,1)$ when $\mu=0$.

\begin{proof}[Proof of Lemma \ref{lem:NumEvOd}.]
    Let $\gamma_n$ be a Dirichlet eigenfunction of (\ref{eq:OntheInt2}) with eigenvalue $\tau_n$. Then $\gamma_{k}$ and $\gamma_{k+1}$ have $k-1$ and $k$ interior zeros in $(0,1)$, dividing $[0,1]$ into $k$ and $k+1$ nodal domains respectively. 
    
    \quad First, suppose $\mu<\tau_{k+1}$ for some $k\geq 0$. Applying the Sturm Comparison Theorem \cite{CL55} to $\xi^e,\xi^o$ with $\gamma_{k+1}$ implies that $\gamma_{k+1}$ has a zero in every interior nodal domain of $\xi^e,\xi^o$. This does not include the intervals on each end of $(0,1)$ as $\xi^e,\xi^o$ do not vanish at the endpoints. Thus,
    \begin{equation}
        N_e-1, N_o-1\leq k. \nonumber
    \end{equation}
    If $0\leq \mu<\tau_1$, so that $k=0$, this immediately implies that $N_e=0$ and $N_o=1$ because both quantities are nonnegative while $N_e$ is even and $N_0$ is odd. Meanwhile, if $\tau_k<\mu$ for some $k\geq 1$, applying the Sturm Comparison Theorem to $\xi^e,\xi^o$ with $\gamma_{k}$ implies that both $\xi^e$ and $\xi^o$ have a zero in every nodal domain of $\gamma_k$. Thus,
    \begin{equation}
        N_e, N_o\geq k. \nonumber
    \end{equation}
    Again using the fact that $N_e$ is even and $N_o$ is odd, this produces the desired result.
\end{proof}

\quad Lemma \ref{lem:NumEvOd} illustrates how the nodal sets of $\xi^e,\xi^o$ are governed by the eigenvalue $\mu$. To determine the splitting of the eigenvalue branches $\lambda^{\eps,e},\lambda^{\eps,o}$ originating from $\mu$, we need to consider the variation given by $\Theta_\mu$ in (\ref{eq:ThetaEq}).

\begin{lemma}\label{lem:ZeroCount}
    $N_e>N_o$ if and only if $\Theta_{\mu}\left(\phi^e(p_0),\phi^e(p_1)\right)>\Theta_{\mu}\left(\phi^o(p_0),\phi^o(p_1)\right)$.
\end{lemma}

\quad In combination with Lemma \ref{lem:NumEvOd}, this statement allows us to distinguish between the upper and lower eigenvalue branches using only knowledge of the eigenvalue $\mu$ and in a manner consistent with Remark \ref{rem:DifferingSigns}. In Section \ref{sec:NDperturbed}, we make use of Theorem \ref{thm:EigFnVar} to compare the number of nodal domains produced respectively by $\varphi^{\eps,e},\varphi^{\eps,o}$ in $R_\epsilon$ with $N_e, N_o$, extending Lemma \ref{lem:ZeroCount} to a statement about nodal domains in $\Omega_\epsilon$.

\begin{rem}\label{rem:BananaSplit}
    Because $N_e$ and $N_o$ differ in parity, Lemma \ref{lem:ZeroCount} implies that at most one of $\Theta_\mu\left(\phi^e(p_0),\phi^e(p_1)\right)$ and $\Theta_\mu\left(\phi^o(p_0),\phi^o(p_1)\right)$ can be zero. This occurs precisely when $\mu$ is a Neumann eigenvalue of \eqref{eq:OntheInt2}. 
\end{rem}

\begin{proof}[Proof of Lemma \ref{lem:ZeroCount}.]
    Without loss of generality, suppose that $\xi^e(0)=\xi^o(0)$ is positive. Between $\xi^e,\xi^o$, the function with the most zeros in the interval is precisely the one that has a zero closest to the origin. Let $x_0\in(0,1)$ denote the smallest zero of $\xi^e,\xi^o$ so that both functions are positive on $[0,x_0)$. Given that each function solves the Sturm-Liouville differential equation, we can write
    \begin{equation}
        -\left(g{\xi^{e}} '\right)'\xi^o=\mu g\xi^e\xi^o \quad \textrm{and} \quad -\left(g{\xi^o}'\right)'\xi^e=\mu g\xi^o\xi^e. \nonumber
    \end{equation}
    Using integration by parts over $(0,x_0)$ yields:
    \begin{align}\label{eq:IntbyParts}
        W\left(\xi^e,\xi^o\right)(x_0)=\frac{1}{2g(x_0)}\left(\Theta_{\mu}\left(\phi^e(p_0),\phi^e(p_1)\right)-\Theta_{\mu}\left(\phi^o(p_0),\phi^o(p_1)\right)\right)
    \end{align}
    where $W\left(\xi^e,\xi^o\right)(x_0)$ denotes the Wronskian of $\xi^e,\xi^o$ at $x_0$. If $\xi^e(x_0)=0$, then ${\xi^e}'(x_0)<0$ and $W\left(\xi^e,\xi^o\right)(x_0)$ is positive. Alternatively, if $\xi^o(x_0)=0$, then ${\xi^o}'(x_0)<0$ and $W\left(\xi^e,\xi^o\right)(x_0)$ is negative, meaning
    \begin{equation}
        N_e>N_o \quad \textrm{if and only if} \quad \xi^e(x_0)=0 \quad \textrm{if and only if} \quad W\left(\xi^e,\xi^o\right)(x_0)>0. \nonumber
    \end{equation}
    Because $g$ is positive, (\ref{eq:IntbyParts}) concludes the proof.
\end{proof}

\quad Importantly, Assumptions \ref{ass:eig} and \ref{ass:phi} are sufficient in guaranteeing that the first-order eigenvalue variations (\ref{eq:ThetaEq}) for $\lambda^{\eps,e},\lambda^{\eps,o}$ are not equal. In combination, Lemmas \ref{lem:NumEvOd} and \ref{lem:ZeroCount} immediately provide the following result, which highlights how the position of the eigenvalue $\mu$ relative to the Dirichlet eigenvalues of (\ref{eq:OntheInt2}) determines the branch variation in (\ref{eq:ThetaEq}).

\begin{prop} \label{prop:index-main}
Let $N_e$, $N_o$ denote the number of zeros of $\xi^e$ and $\xi^o$ in $(0,1)$, and let $k\geq1$ be so that $\tau_k<\mu<\tau_{k+1}$ (or $k=0$ if $\mu<\tau_1$). Then, 
\begin{enumerate}
    \item [(i)] If $k$ is even, then
    \begin{equation}
        N_e = k, \quad N_o = k+1, \quad \textrm{and} \quad \Theta_\mu(\phi^e(p_0),\phi^e(p_1)) < \Theta_\mu(\phi^o(p_0),\phi^o(p_1)); \nonumber
    \end{equation}
    \item [(ii)] If $k$ is odd, then
    \begin{equation}
        N_e = k+1, \quad N_o = k, \quad \textrm{and} \quad \Theta_\mu(\phi^e(p_0),\phi^e(p_1)) > \Theta_\mu(\phi^o(p_0),\phi^o(p_1)). \nonumber
    \end{equation}
\end{enumerate}
\end{prop}

\quad Let index$(\mu,\Omega_L)$ denote the index of the eigenvalue $\mu$ in the Neumann spectrum of $\Omega_{L}$ and index$(\lambda^{\eps,e},\Omega_\eps)$, index$(\lambda^{\eps,o},\Omega_\eps)$ the index of the eigenvalues $\lambda^{\eps,e}$, $\lambda^{\eps,o}$ respectively in the Neumann spectrum of $\Omega_\eps$. Alongside Theorem \ref{thm:EigVar}, Proposition \ref{prop:index-main} and  \cite[Theorem $2.2$]{A95}, as highlighted in Remark \ref{rem:order}, imply the following statement about eigenvalue index.

\begin{cor} \label{cor:index-main}
There exists $\eps_0>0$ such that for $0<\epsilon<\eps_0$,
\begin{enumerate}
    \item [(i)] If $k$ is even, then
    \begin{equation}
        \emph{index}\left(\lambda^{\eps,e},\Omega_\eps\right) = 2\:\emph{index}\left(\mu,\Omega_L\right) + k-1 \quad \textrm{and} \quad \emph{index}\left(\lambda^{\eps,o},\Omega_\eps\right) = 2\:\emph{index}\left(\mu,\Omega_L\right) + k; \nonumber
    \end{equation}
    \item [(ii)] If $k$ is odd, then
    \begin{equation}
        \emph{index}\left(\lambda^{\eps,e},\Omega_\eps\right) = 2\:\emph{index}\left(\mu,\Omega_L\right) + k \quad \textrm{and} \quad \emph{index}\left(\lambda^{\eps,o},\Omega_\eps\right) = 2\:\emph{index}\left(\mu,\Omega_L\right) + k-1. \nonumber
    \end{equation}
\end{enumerate}
\end{cor}

\section{Nodal Domains of the Perturbed Eigenfunctions}\label{sec:NDperturbed}

\quad In this section, we determine the qualitative behavior of the nodal sets of $\varphi^{\eps,e}$ and $\varphi^{\eps,o}$ in the perturbed domain $\Omega_\epsilon$. Continuing to work under Assumptions \ref{ass:eig} and \ref{ass:phi}, we use Theorem \ref{thm:EigFnVar} to establish a count for the number of nodal domains of each eigenfunction and establish the following proposition.

\begin{prop} \label{prop:nodal-main}
Let $\llbracket \phi,\Omega_L\rrbracket$ denote the number of nodal domains of $\phi$ in $\Omega_L$, and likewise let $\llbracket \varphi^{\eps,e},\Omega_\eps\rrbracket$, $\llbracket \varphi^{\eps,o},\Omega_\eps\rrbracket$ denote the respective number of nodal domains of $\varphi^{\eps,e}$ and $\varphi^{\eps,o}$ in $\Omega_\eps$. Finally, let $k\geq0$ be as in Proposition \ref{prop:index-main}. Then there exists $\eps_0>0$ such that for all $0<\eps<\eps_0$, we have:
\begin{enumerate}
    \item [(i)] If $k$ is even, then
    \begin{equation}
        \llbracket \varphi^{\eps,e},\Omega_\eps\rrbracket \leq  2\llbracket \phi,\Omega_L\rrbracket + k-1  \quad \textrm{and} \quad \llbracket\varphi^{\eps,o},\Omega_\eps\rrbracket \leq  2\llbracket \phi,\Omega_L\rrbracket + k; \nonumber
    \end{equation}
    \item [(ii)] If $k$ is odd, then
    \begin{equation}
        \llbracket \varphi^{\eps,e},\Omega_\eps\rrbracket\leq  2\llbracket \phi,\Omega_L\rrbracket + k  \quad \textrm{and} \quad \llbracket\varphi^{\eps,o},\Omega_\eps\rrbracket \leq  2\llbracket \phi,\Omega_L\rrbracket + k-1. \nonumber
    \end{equation}
\end{enumerate}
Moreover, if $\phi$ has no crossings in its nodal set in $\Omega_L$, then we have equality in all of the above statements.
\end{prop}

\quad Combining this result with Corollary \ref{cor:index-main} in Section \ref{sec:NDoriginal} implies Theorem \ref{thm:Main}. To prove Proposition \ref{prop:nodal-main}, we carry out an analysis that is performed locally according to a partition of $\Omega_{\eps}$. A version of this partition was also used in \cite{BCM-pleijel}, when providing an upper bound on the number of nodal domains of the Neumann eigenfunctions of a family of chain domains (which includes the symmetric dumbbell domains under consideration in this paper). Properties of the partition and a partition of unity adapted to it depend on certain geometric properties of the domain that are independent of the neck width $\eps$. We define these quantities in terms of the base domain $\Omega_1$ (that is the  dumbbell domain $\Omega_\eps$ with $\eps = 1$) and $L^* \eqdef \text{Length}(\pa\Omega_0) + \text{Length}(\{y = g(x):x\in[0,1]\}) + 1$. Note that $L^*$ provides an upper bound on the perimeter of $\Omega_\eps$ for all $0<\eps<1$.  We now define the four geometric quantities:
\begin{itemize}
     \item \emph{Normalized curvature constant, $\kappa^*$}: We define $\kappa^*>0$ so that the curvature of each smooth segment of $\pa \Omega_1$ is bounded above by $\kappa^*/L^*$.
    \item \emph{Vertex control constant, $\delta^*$}:  The constant $\delta^*>0$ is defined so that the following holds. For each vertex  $p \in \pa \Omega_1$, we have that the connected component of  $ \pa\Omega_1$ in the disc centered at $p$ of radius $L^*\delta^*$ consists of two smooth curves joined at $p$, and after possibly rotating, we may assume that their tangent lines at $p$ agree with the lines $\theta = \tfrac{\pi}{2} \pm\tfrac{\theta_0}{2}$ for some $0 < \theta_0 < \pi$. These are graphs with respect to the $x$-axis, contained within the lines $\theta = \tfrac{\pi}{2} \pm\tfrac{\theta_0}{4}$ and $\theta = \tfrac{\pi}{2} \pm\tfrac{3\theta_0}{4}$, and with slope bounded by $1/\delta^*$.
        \item  \emph{Normalized cut-distance constant, $\sigma^*$}: The constant $\sigma^*$ is defined so that, for $\eta\leq L^*\delta^*$, the cut-distance, given by
    \begin{align*}
    \inf_{u\in K_{\eta}}\sup_{\delta>0}\left\{\delta: \;\text{dist}\Big(\gamma(u)+sn(u),\pa\Omega_0\Big)=s \; \text{for all} \;s\in[0,\delta] \Big.\right\},
\end{align*}
is bounded from below by $\sigma^*\eta$. Here, $\gamma:K_{\eta}\to\mathbb{R}^2$ is a parameterization of $\pa \Omega_0$ with the parts of $\pa\Omega_0$ in the discs of radius $\eta$ centered at each vertex of $\pa\Omega_0$ excluded, for a union of intervals $K_{\eta}$, and $n(u)$ is the inward unit normal at $\gamma(u)\in\pa\Omega_0$. (Note that this constant is denoted by $\tau^*$ in \cite{BCM-pleijel}.)
   \item \emph{Neck-width constant, $w^*$}: The constant $w^*$ is defined to provide control on the widths and regularity of the neck, so that $\min_{x\in[0,1]}g(x) \geq w^*$ and    $\max_{x\in[0,1]}|g'(x)| \leq  1/w^*$. 
\end{itemize}

With these constants in tow, we can now define the partition and describe some of its properties.
\begin{lemma}\label{lem:Partition}
    There exists a small constant $\delta_0>0$, depending on the above geometric constants, such that for any $\delta<\delta_0$ and $\epsilon<\delta/4$, $\Omega_\eps$ can be partitioned as
    \begin{equation}
        \Omega_\epsilon=\bigcup_{j=0}^4 \Omega_j^\delta(\epsilon) \nonumber
    \end{equation}
    where the following holds:
\begin{itemize}
    \item[1)] $\Omega_{2}^{\delta}(\eps)$ is the intersection of $\Omega_{\eps}$ with discs of radius $\delta$ centered at each vertex of $\p\Omega_{L}$ and $\p\Omega_R$.
    \item[2)] $\Omega_{4}^{\delta}(\eps)$ is the intersection of $\Omega_{\eps}$ with discs of radius $\delta$ centered at $p_0 = (0,0)$ and $p_1=(1,0)$ at either end of the neck.
    \item[3)] $\Omega_{3}^{\delta}(\eps)$ is  the remainder of the neck $R_{\eps}\backslash\Omega_{4}^{\delta}(\eps)$.
    \item[4)] $\Omega_{1}^{\delta}(\eps)$ is  $\{(x,y)\in\Omega_L\cup\Omega_R:\emph{dist}((x,y),\pa\Omega_L\cup\pa\Omega_R)<\tfrac{3}{4}\sigma^*\delta\}\backslash \left(\Omega_{2}^{\delta}(\eps)\cup\Omega_{4}^{\delta}(\eps)\right)$, and $\Omega_{0}^{\delta}(\eps)$ is given by
    \begin{align*}
        \Omega_{0}^{\delta}(\eps) = \Omega_{\eps}\backslash \bigcup_{j=1}^{4}\Omega_{j}^{\delta}(\eps).
    \end{align*}
\end{itemize}
Moreover, there exists a constant $C^*$ such that there exists a partition of unity $\chi_j^\delta$ adapted to this partition,  with gradient bounded almost everywhere by $C^*\delta^{-1}$. Here $\sigma^*$ is the normalized cut-distance constant, and $C^*$ depends only on the four geometric constants defined above.
\end{lemma}

\quad Motivated by \cite[Definition $3.1$]{BCM-pleijel}, this partition separates neighborhoods of the smooth parts of the boundaries of $\Omega_L,\Omega_R$, their vertices, and the neck $R_\epsilon$. Figure \ref{fig:Partition} below features an example of such a partition when $\Omega_L$ and $\Omega_R$ are square domains. For a proof of Lemma \ref{lem:Partition}, see \cite[Section $3.2$]{BCM-pleijel}, and in particular \cite[Lemma $3.2$]{BCM-pleijel}.

\begin{figure}[H]
    \centering
    \includegraphics[scale=0.7]{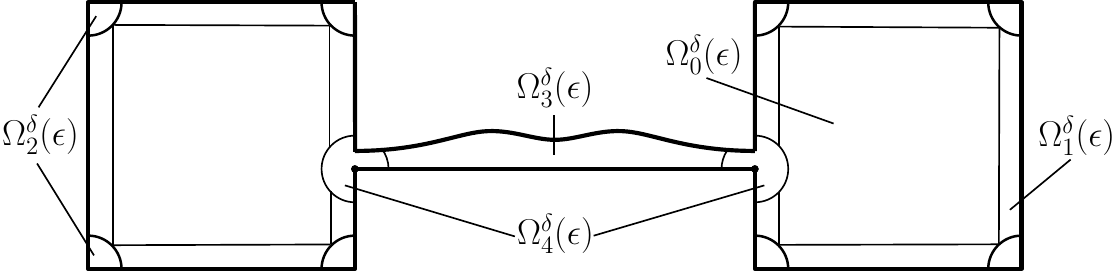}
    \caption{The partition of $\Omega_\eps$ given in Lemma \ref{lem:Partition}.}
    \label{fig:Partition}
\end{figure}

\quad The constant $\delta_0$ can also be chosen such that the connected components of $\p\Omega_L$ and $\p\Omega_R$ in  $\overline{\Omega_1^\delta(\eps)}$ can be straightened as in \cite[Lemma $3.1$]{BCM-pleijel}. Namely, if one component of $\p\Omega_L\cap\overline{\Omega_1^\delta(\epsilon)}$ has the parametrization $\{\gamma(s):s\in I\}$ and if $n(s)$ denotes the inward-pointing unit normal of $\p\Omega_L$ at $\gamma(s)$, then, for $0<\delta<\delta_0$,
\begin{align} \label{eqn:straightening}
    F:I\times\left[0,\tfrac{3}{4}\sigma^*\delta\right]\to \Omega_L, \qquad (s,t)\mapsto (x,y) = \gamma(s) +tn(s) 
\end{align}
is a diffeomorphism onto its image, with Jacobian bounded above and below. See \cite[Section $3$]{lena-pleijel} for more details on this construction. The top boundary of $R_\epsilon$ can also be straightened by a diffeomorphism, allowing us to establish the following result for nodal domains intersecting $\Omega_4^\delta(\epsilon)\cap R_\eps$.

\begin{lemma}\label{lem:Omega4}
    Let $\delta_0$ be as in Lemma \ref{lem:Partition}. There exists a positive constant $\tilde{\delta}$, independent of $\epsilon$ and $\delta$, such that if for a nodal domain $D$ of $\varphi^{\epsilon,e}$ we have $\norm{\varphi^{\epsilon,e}}_{L^2\left(\Omega_4^\delta(\epsilon)\cap D\cap R_\epsilon\right)}\geq \tfrac{1}{2}\norm{\varphi^{\epsilon,e}}_{L^2(D)}$, for some $0<\delta<\delta_0$, then the set
    \begin{equation}
        \left\{x\in\mathbb{R}:(x,y)\in D\cap R_\eps \text{ for some } y\in \mathbb{R}\right\} \nonumber 
    \end{equation}
    has size at least $\tilde{\delta}$. The same statement holds for the nodal domains of $\varphi^{\epsilon,o}$.
\end{lemma}

\quad This statement provides a lower bound on the projection of any such nodal domain onto the $x$-axis. In particular, it implies that neither eigenfunction produces a nodal domain contained within small neighborhoods near the edges of the neck. 

\begin{proof}[Proof of Lemma \ref{lem:Omega4}.] We demonstrate the proof for $\varphi^{\epsilon,e}$ as an analogous argument holds for $\varphi^{\epsilon,o}$.  Following the proof of \cite[Lemma $6.2$]{BCM-pleijel}, suppose $D$ is a nodal domain of $\varphi^{\epsilon,e}$ that intersects $ \Omega_4^\delta(\epsilon)\cap R_\epsilon$. By multiplying $\varphi^{\epsilon,e}$ by a constant, which leaves the nodal domains invariant, we may assume that $\norm{\varphi^{\epsilon,e}}_{L^2(D)}=1$. Let 
\begin{equation}\label{eq:NeckDiffeo}
F(x,y)\eqdef\left(x,yg(x)^{-1}\right): R_\epsilon \to [0,1]\times[0,\epsilon]
\end{equation}
denote a diffeomorphism which flattens the boundary and preserves vertical lines. Similarly to the diffeomorphism featured in (\ref{eqn:straightening}), note that the Jacobian of $F$ is bounded by the neck-width constant $w^*$, independently of $\epsilon$. Then $V=F\left(D\cap R_\epsilon\right)$ is a subset of a strip of width $\epsilon$ and $v=\varphi^{\epsilon,e}\circ F^{-1}\in H^1(V)$ satisfies $$\int_V v^2 \geq c>0 \quad \textrm{and} \quad \int_V |\nabla v|^2\leq C\mu$$ for some constants $c,C$ that are independent of $\epsilon$ and $\delta$.

\quad Considering, without loss of generality, the region of $\Omega_4^\delta(\epsilon)$ near $p_0$, there exists some $x^*$ with $0\leq x^*<\delta<\delta_0<1/4$ such that
\begin{equation}
    \int_{V(x^*)} v^2(x^*,y)dy\geq c\delta^{-1}> 4c, \nonumber
\end{equation}
where $V(x^*) = \{y:(x^*,y)\in V\}$ is the cross-section of $V$ at $x=x^*$. We extend $v$ identically by zero outside of $V$, noting that this extension is not in $H^1(\mathbb{R}^2)$ because $v$ is not identically zero along the vertical line $\{x=0\}$ or the horizontal lines $\{y=0\}$ or $\{y=\eps\}$. However, if we label this extension as $w$, then
\begin{equation}
    w(\cdot, y)\in H^1([x^*,\delta]) \quad \textrm{and} \quad \int_{V(x^*)}\int_{x^*}^{\delta} \left|\nabla w\right|^2 dxdy\leq C\mu. \nonumber
\end{equation}
Using the fundamental theorem of calculus to write
\begin{equation}
    w(x,y)=w(x^*,y)+\int_{x^*}^x \p_tw(t,y)dt, \nonumber
\end{equation}
we can then express
\begin{align}
    w^2(x,y)&=w^2(x^*,y)+2w(x^*,y)\int_{x^*}^x \p_t w(t,y)dt+\left(\int_{x^*}^x \p_tw(t,y)dt\right)^2 \nonumber \\
    &\geq w^2(x^*,y)-2\left|w(x^*,y)\right|\left|\int_{x^*}^x \p_tw(t,y)dt\right|+\left(\int_{x^*}^x \p_tw(t,y)dt\right)^2 \nonumber \\
    &\geq \frac{1}{2}w^2(x^*,y)-\left(\int_{x^*}^x \p_tw(t,y)dy\right)^2 \nonumber
\end{align}
by Young's inequality. According to Cauchy-Schwarz, this implies
\begin{equation}\label{eq:wLowBd}
    w^2(x,y)\geq \frac{1}{2}w^2(x^*,y)-|x-x^*|\int_{x^*}^x \left|\p_t w(t,y)\right|^2 dt.
\end{equation}
\quad Integrating (\ref{eq:wLowBd}) over $y\in V(x^*)$ yields
\begin{align}
    \int_{V(x^*)} w^2(x,y)dy&\geq \frac{1}{2}\int_{V(x^*)}w^2(x^*,y)dy-|x-x^*|\int_{V(x^*)} \int_{x^*}^x \left|\p_tw(t,y)\right|^2 dtdy \nonumber \\ &\geq \frac{1}{2}\int_{V(x^*)} v^2(x^*,y)dy-|x-x^*| \int_{V(x^*)} \int_{x^*}^\delta |\nabla w|^2 dxdy\nonumber \\ &\geq 2c-C\mu|x-x^*| \nonumber
\end{align}
and hence    
\begin{equation}
    \int_{y_1}^{y_2}w^2(x,y)dy\geq c>0 \nonumber
\end{equation}
for all $x\geq x^*$ such that $|x-x^*|\leq c(C\mu)^{-1}\defeq \tilde{\delta}$. Note that $\tilde{\delta}$ is independent of $\eps$ and $\delta$. Because $w$ is equal to zero outside of $V$, this implies that $V$ has nonempty intersection with $\{x=x^*+\tilde{\delta}\}$. Because $F$ preserves vertical lines and $D$ is connected, this provides the desired result.
\end{proof}

\quad In the remainder of the neck, we use the estimates in Theorem \ref{thm:EigFnVar}  to establish a nodal count for the perturbed eigenfunctions in $R_\epsilon$. From Section \ref{sec:NDoriginal}, we retain the notation $\xi^e=\xi_\mu^{\phi^e}$ and $\xi^o=\xi_\mu^{\phi^o}$ as solutions to (\ref{eq:OntheInt2}) with boundary conditions determined by $\phi^e,\phi^o$.

\begin{lemma} \label{lem:SameNumZeros}
Let $N_e,N_o$ denote the number of zeros in $(0,1)$ of $\xi^e,\xi^o$ respectively. There exists a constant $\delta_0>0$, such that for each $0<\delta<\delta_0$, we can find $\eps_0>0$ so that, if $0<\eps<\eps_0$, then
\begin{enumerate}
    \item the nodal set of $\varphi^{\epsilon,e}$ in $\Omega_3^\delta(\epsilon)$ consists of $N_e$ disjoint curves, each intersecting the top and bottom boundaries of $R_\eps$;

    \item the nodal set of $\varphi^{\epsilon,o}$ in $\Omega_3^\delta(\epsilon)$ consists of $N_o$ disjoint curves, each intersecting the top and bottom boundaries of $R_\eps$.
\end{enumerate}
Further, $\varphi^{\eps,e}$ and $\varphi^{\eps,o}$ have the same sign as $\phi^e(p_0)$ and $\phi^o(p_0)$, respectively, on the left side of $\p\Omega_3^\delta(\eps)\cap R_\eps$ and the same sign as $\phi^e(p_1)$ and $\phi^o(p_1)$, respectively, on the right side of $\p\Omega_3^\delta(\eps)\cap R_\eps$.
\end{lemma}

\quad This result states that for small enough $\eps$, the number of nodal domains produced by $\varphi^{\epsilon,e}$ and $\xi^e$, or analogously by $\varphi^{\eps,o}$ and $\xi^o$, in the neck matches. In combination with Lemma \ref{lem:ZeroCount}, this implies that between $\lambda^{\eps,e}$ and $\lambda^{\eps,o}$, the upper eigenvalue branch corresponds precisely to the eigenfunction that produces more nodal domains in $R_\epsilon$.

\begin{proof}[Proof of Lemma \ref{lem:SameNumZeros}.] We demonstrate the proof for $\varphi^{\epsilon,e}$ as an analogous argument holds for $\varphi^{\epsilon,o}$. We first show that $\varphi^{\eps,e}$ cannot have a nodal domain contained in a thin subset of the neck: Suppose $\varphi^{\epsilon,e}$ has a nodal domain $D$ contained in a subset of $R_\epsilon$ of width $c_1>0$, i.e. $D\subseteq R_\epsilon\cap \{(x,y): x_0\leq x\leq x_0+c_1\}$ for some $x_0$. Let
\begin{equation}
    F:R_\epsilon\to [0,1]\times[0,\epsilon] \nonumber
\end{equation}
be as in (\ref{eq:NeckDiffeo}), and let $\tilde{F}$ be the composition of $F$ with a reflection across the top boundary of the neck, and a gluing of the resulting boundary components. Then $\textrm{Area}\left(\tilde{F}(D)\right)\leq 2c_1\epsilon$, and $\tilde{F}(D)$ is contained in a flat cylinder  of circumference $2\epsilon$. According to a Faber-Krahn inequality for the cylinder, \cite[Lemma $6.3$]{BCM-pleijel}, for small enough $\epsilon$,
\begin{equation}\label{eq:AreaLwrBd}
    \textrm{Area}\left(\tilde{F}(D)\right)\geq 2\epsilon \lambda_1(\mathbb{D})^{1/2} \lambda_1(\tilde{F}(D))^{-1/2}
\end{equation}
where $\lambda_1(\cdot)$ denotes the smallest Dirichlet eigenvalue of the set and $\mathbb{D}$ is the unit disc. By using $\varphi^{\eps,e}\circ\tilde{F}$ in the variational formulation of the first eigenvalue, we have $\lambda_1(\tilde{F}(D))=C\mu$, and hence (\ref{eq:AreaLwrBd}) states that
\begin{equation}
    c_1\geq C^{-1/2}\lambda_1(\mathbb{D})^{1/2}\mu^{-1/2}. \nonumber
\end{equation}
Thus, if we let $c_1^*<C^{-1/2}\lambda_1(\mathbb{D})^{1/2}\mu^{-1/2}$, then no nodal domain of $\varphi^{\epsilon,e}$ can exist in $R_\epsilon\cap \{(x,y):x_0\leq x\leq x_0+c_1^*\}$. Because $c_1^*$ can be chosen independent of $\epsilon$, this provides a lower bound on the size of the projection onto the $x$ axis of any nodal domain in the neck.

\quad We now show that, for sufficiently small $\eps>0$, it is possible to partition $R_\epsilon$ by a collection of vertical lines, separated by a distance no more than $c_1^*$, such that $\varphi^{\epsilon,e}$ and $\xi^e$ have the same sign along each line. According to Theorem \ref{thm:EigFnVar}, there exists some positive constant $C$ such that
\begin{equation}\label{eq:H1restrict}
    \int_U \int_0^{\epsilon g(x)}\left|\varphi^{\epsilon,e}-\xi^e\right|^2+\left|\nabla(\varphi^{\epsilon,e}-\xi^e)\right|^2 dydx=o(\epsilon)<C\epsilon
\end{equation}
for any $U\subseteq (0,1)$ with size $|U|$. We let 
\begin{equation}
    f_\epsilon(x)\eqdef\int_0^{\epsilon g(x)}\left|\varphi^{\epsilon,e}-\xi^e\right|^2+\left|\nabla(\varphi^{\epsilon,e}-\xi^e)\right|^2 dy\geq 0, \nonumber
\end{equation}
and define $\tilde{U}=\{x\in U:f_\eps(x) \leq 2 C|U|^{-1}\eps\}$. Then, $$\int_U f_\epsilon(x)dx\geq \int_{U\backslash \tilde{U}}f_\epsilon(x)dx>2C|U|^{-1}\epsilon|U\backslash \tilde{U}|,$$ and so by \eqref{eq:H1restrict}, $|U\backslash \tilde{U}| < \tfrac{1}{2}|U|$. As a consequence, $|\tilde{U}| >\tfrac{1}{2}|U|$. 
By partitioning $(0,1)$ into sets $U_j$ with size less than $c_1^*/2$, we can therefore recover a finite sequence of points $\{x_j^*\}$ such that
\begin{equation}
    |x_i^*-x_j^*|\leq c_1^* \quad \textrm{and} \quad f_\epsilon(x_j^*)\leq 2C|U_j|^{-1}\epsilon. \nonumber
\end{equation}
We can also choose these points so that $\xi^e(x_j^*)\neq 0$ for all $j$. Because the sequence is finite, we can bound $f_\eps(x_j^*)$ above by a uniform multiple of $\epsilon$.

\quad For $x\in(0,1)$, let $R_\epsilon(x)$ denote the cross-section of $R_\epsilon$ at $x$. Because $\xi^e$ depends only on $x$, we claim that $\varphi^{\epsilon,e}$ and $\xi^e$ share the same sign along each $R_\epsilon(x_j^*)$ for sufficiently small $\eps>0$. For the sake of contradiction, suppose that for some $j$, $\xi^e(x_j^*)$ is negative and $\varphi^{\eps,e}(x_j^*,y^*)$ is positive for some $y^*$. Set $c_2^*=\min_j |\xi^e(x_j^*)|$ and $$p(y)\eqdef \varphi^{\eps,e}(x_j^*,y) - \xi^e(x_j^*)$$ so that $p(y^*)>c_2^*$ and
\begin{equation}
    \left|p(y)-p(y^*)\right|^2=\left|\int_y^{y^*} p'(t)dt\right|^2\leq |y-y^*| \int_0^{\epsilon g(x_j^*)}|p'(t)|^2 dt\leq \epsilon f_\epsilon(x_j^*). \nonumber
\end{equation}
This implies that for $\epsilon$ sufficiently small, $p(y)>c_2^*/2$ for all $y$ in the cross-section. However, this contradicts (\ref{eq:H1restrict}) and hence $\varphi^{\epsilon,e}$ and $\xi^e$ have the same sign along the cross section of $R_\epsilon$ at all the $x_j^*$. 

\quad By construction, no nodal domain of $\varphi^{\epsilon,e}$ exists between the vertical lines $R_\epsilon(x_j^*)$ and $R_\epsilon(x_{j+1}^*)$. By continuity, $\varphi^{\epsilon,e}$ vanishes at some points in this region if and only if $\xi^e(x)$ vanishes for some $x\in(x_j^*,x_{j+1}^*)$. For $\delta<\delta_0$, with $\delta_0$ as in Lemma \ref{lem:Partition}, we can reduce $c_1^*$ further if necessary to ensure that $c_1^*<\delta$, and then for $\delta$ small enough we can conclude that the nodal set of $\varphi^{\epsilon,e}$ in $\Omega_3^\delta(\epsilon)$  features $N_e$ distinct curves that intersect the top and bottom boundary of $R_\epsilon$ and that $\varphi^{\eps,e}$ has the same sign as $\xi^e$ along $\p\Omega_3^\delta(\eps)\cap R_\eps$. Taking $\delta$ small relative to the spacing between the zeros of $\xi^e$, we have that $\varphi^{\eps,e}$ then has the same sign as $\phi^e(p_0)$ along the left component of $\p\Omega_3^\delta(\eps)\cap R_\eps$ and the same sign as $\phi^e(p_1)$ along the right component of $\p\Omega_3^\delta(\eps)\cap R_\eps$.
\end{proof}

\quad Lemmas \ref{lem:Omega4} and \ref{lem:SameNumZeros} provide information about the nodal sets of $\varphi^{\eps,e}$ and $\varphi^{\eps,o}$ within the neck $R_{\eps}$. We now turn to understanding their nodal sets within the bulk domains $\Omega_L$ and $\Omega_R$. To do this, we compare the nodal sets of the perturbed eigenfunctions to those of $\phi^e$ and $\phi^o$ using Theorem \ref{thm:EigFnVar}. We first show that if the restrictions of $\varphi^{\eps,e}$, $\varphi^{\eps,o}$ to a particular nodal domain has a sufficiently large proportion of its $L^2$-mass in $\Omega_L$ or $\Omega_R$, then we can obtain a lower bound on the area of the nodal domain, independent of $\eps$.

\begin{lemma} \label{lem:BCM}
  There exist constants $c^*>0$, $\eps_0>0$, independent of $\eps>0$, such that if $0<\eps<\eps_0$ and if $D$ is a nodal domain of $\varphi^{\eps,e}$ with $\norm{\varphi^{\eps,e}}_{L^2\left(D\cap(\Omega_L\cup\Omega_R)\right)}\geq \tfrac{1}{2}\norm{\varphi^{\eps,e}}_{L^2(D)}$, then $$\emph{Area}(D)\geq c^*.$$ The same estimate holds for $\varphi^{\eps,o}$ and its nodal domains.
\end{lemma}

\begin{rem} \label{rem:BCM}
    The proof of Lemma \ref{lem:BCM} relies only on the fact that $\varphi^{\eps,e}$ is a Neumann eigenfunction on $\Omega_\eps$ with a bound on its eigenvalue that is independent of $\eps$. In particular, it does not make use of Assumption \ref{ass:phi}.
\end{rem}

\begin{proof}[Proof of Lemma \ref{lem:BCM}.]
We first note that from \cite[Lemma $3.3$]{BCM-pleijel}, setting $u = \varphi^{\eps,e}$ and $\lambda = \lambda^{\eps,e}$, we have
\begin{align} \label{eqn:Green}
\int_D|\nabla u|^2 = \lambda\int_Du^2,
\end{align}
and this also holds for the alternative choice of $u=\varphi^{\eps,o}$ and $\lambda=\lambda^{\eps,o}$. As in \cite[Section $3.3$]{BCM-pleijel}, we use the partition in Lemma \ref{lem:Partition} to define bulk, boundary, corner, and neck nodal domains. We fix $\delta$ with $0<\delta<\delta_0$, with $\delta_0$ chosen as in Lemma \ref{lem:Partition}. For $u = \varphi^{\eps,e}$, set $u_j = \chi_j^\delta u$. We say that the nodal domain $D$ of $u$ is a
\begin{itemize}
        \item \textit{bulk nodal domain} if $\norm{u_0}_{L^2(D)} \geq \tfrac{1}{8}\norm{u}_{L^2(D)}$;
        \item \textit{boundary nodal domain} if $\norm{u_1}_{L^2(D)} \geq \tfrac{1}{8}\norm{u}_{L^2(D)}$;
        \item \textit{corner nodal domain} if $\norm{u_2}_{L^2(D)} \geq \tfrac{1}{8}\norm{u}_{L^2(D)}$ or $\norm{u_4}_{L^2(D\cap(\Omega_L\cup\Omega_R))} \geq \tfrac{1}{8}\norm{u}_{L^2(D)}$;
        \item \textit{neck nodal domain} if $\norm{u_3}_{L^2(D)} \geq \tfrac{1}{8}\norm{u}_{L^2(D)}$ or $\norm{u_4}_{L^2(D\cap R_{\eps})} \geq \tfrac{1}{8}\norm{u}_{L^2(D)}$.
\end{itemize}
Note that every nodal domain of $u$ must fall into at least one of these categories. Because we consider a nodal domain satisfying $$\norm{u}_{L^2\left(D\cap \left(\Omega_L\cup \Omega_R\right)\right)}\geq \tfrac{1}{2}\norm{u}_{L^2(D)},$$ $D$ must be either a bulk, boundary, or corner nodal domain. By multiplying $u$ by an appropriate constant, which does not alter the nodal domains, we may assume that $\norm{u}_{L^2(D)} = 1$.
   
\quad In each case, we find a set $V$ of area comparable to $D$ and a function $v\in H_0^1(V)$ with a bounded Rayleigh quotient. Then the Faber-Krahn Theorem \cite{faber1923,krahn1925} can be applied to provide a lower bound on the area of $D$ as given in the statement of the lemma. The constants appearing in these estimates given below depend on an upper bound on the eigenvalue $\lambda$, the fixed $\delta$, and the geometric constants stated before Lemma \ref{lem:Partition} (in particular, through the constant $C^*$ in Lemma \ref{lem:Partition}), but crucially are independent of $\eps>0$. As this follows from the proofs featured in \cite[Section $4,5$]{BCM-pleijel} used to count the number of nodal domains, we describe the key ideas in each case. 

\quad If $D$ is a bulk nodal domain, then we follow the argument in \cite[Section $4.1$]{BCM-pleijel}: Since in this case $u_0\in H^1_0(D)$, we use $v= u_0= \chi_0^\delta u$, and $V=D$. From \eqref{eqn:Green}, and the bound on the gradient of $\chi_0^{\delta}$ from Lemma \ref{lem:Partition}, we have that
\begin{align*}
    \int_{D}|\nabla u_0|^2\bigg/\int_Du_0^2 \leq 8^2\int_{D}|\nabla u_0|^2\bigg/\int_D u^2 \leq C_0^*.
\end{align*}
The Faber-Krahn Theorem then gives a lower bound of
\begin{align*}
    \pi\lambda_1(\mathbb{D})\text{Area}(D) \leq C_0^*,
\end{align*}
where $\lambda_1(\mathbb{D})$ is the first Dirichlet eigenvalue of the unit disc. By rearranging this inequality, we have a lower bound on the area of $D$, independent of $\eps$.

\quad If $D$ is a boundary nodal domain, then there exists a subset, $b_{\delta}$, of a smooth component of $\p\Omega_L\cup\p\Omega_R$, contained within $\Omega_1^{\delta}(\eps)$ such that
\begin{align*}
    \int_{\left\{\text{dist}((x,y),b_\delta)<3\sigma^*\delta/4\right\}} u_1^2 \geq c_1^*\int_Du^2 = c_1^*
\end{align*}
for some constant $c_1^*$, depending only on the total number of smooth components of $\p\Omega_L\cup\p\Omega_R$. We use the straightening diffeomorphism from \eqref{eqn:straightening} and apply a reflection of $u_1$ across the straightened boundary in an argument identical to  \cite[Equations $(13)-(17)$]{lena-pleijel}, \cite[Lemma $4.1$]{BCM-pleijel}. We can then conclude that there exists a Lipschitz set $V$, a function $v\in H_0^1(V)$, and a constant $C_1^*$, such that
\begin{align*}
    \text{Area}(V) \leq C_1^*\text{Area}(D) \quad \textrm{and} \quad  \int_{V}|\nabla v|^2\bigg/\int_{V}v^2 \leq C_1^*.
\end{align*}
Applying the Faber-Krahn Theorem again gives a desired lower bound on the area of $D$.

 \quad Finally, if $D$ is a corner nodal domain, we follow \cite[Section $5$]{BCM-pleijel} and split the analysis further into two cases, depending on whether 
 \begin{equation}\label{eq:cases}
     \norm{u_2}_{L^2(D)} \geq \tfrac{1}{8}\norm{u}_{L^2(D)} = \tfrac{1}{8} \quad \textrm{or} \quad \norm{u_4}_{L^2(D\cap(\Omega_L\cup\Omega_R))} \geq \tfrac{1}{8}\norm{u}_{L^2(D)} = \tfrac{1}{8}.
\end{equation}
In the first case of (\ref{eq:cases}), there exists a vertex $v_0$ of $\p\Omega_L\cup\p\Omega_R$ with opening angle $\theta_0$ and a constant $c_2^*$, depending only on the total number of vertices, such that
   \begin{align*}
       \int_{D\cap\{\text{dist}((x,y),v_0)<\delta\}}u_2^2 \geq c_2^*.
   \end{align*}
   Letting $S_{\theta_0}(\delta) = \{(r,\theta)\in\bbR^2:r<\delta, |\theta|<\theta_0/2\}$ be a subset of an exact sector, there exists a set $V\subset S_{\theta_0}(\delta)$ and function $v\in H^1(V)$, with $v\equiv0$ on $\pa V\cap S_{\theta_0}(\delta)$, such that Area$(V)\leq C_2^*$Area$(D)$ and 
   \begin{align} \label{eqn:BCM-corner1}
       \int_V v^2\geq (C_2^*)^{-1}, \qquad \int_V |\nabla v|^2\leq C_2^*.
   \end{align}
These estimates follow by applying a transformation to $u_2$ which translates, rotates, and straightens the smooth components of $\p\Omega_L\cup\p\Omega_R$ meeting at $v_0$. See the proof of \cite[Lemma $5.1$]{BCM-pleijel} for the details of this transformation. As in Lemma \ref{lem:Omega4}, we extend $v$ by zero outside of $V$ to get a function $w$ defined on $\mathbb{R}^2$ but not necessarily in $H^1(\mathbb{R}^2)$. From the first estimate in \eqref{eqn:BCM-corner1}, there exists some $x^*$ with $|x^*|<\delta$ such that
   \begin{align*}
       \int_{V(x^*)} w(x^*,y)^2\,dy\geq (C_2^*)^{-1}\delta^{-1},
   \end{align*}
   where $V(x^*)=\{y:(x^*,y)\in V\}$ is the cross-section of $V$ at $x^*$.  The idea now is to use the second estimate in \eqref{eqn:BCM-corner1} to get the lower bound $\int_{y_1}^{y_2}w(x,y)^2\,dy\geq \tfrac{1}{2}(C_w^*)^{-1}\delta^{-1}$ for nearby $x>x^*$, and use this to show that there exists a cut-off function $\chi$, supported in the interior of $S_{\theta_0}(\delta)$ such that
   \begin{align*}
       \int_V |\chi w|^2 \geq (\tilde{C}_2^*)^{-1}, \qquad \int_V |\nabla(\chi w)|^2 \leq \tilde{C}_2^*,
   \end{align*}
   for a constant $\tilde{C}_2^*$ depending only on $C_2^*$ and $\delta$. This can be achieved by following exactly the proof of \cite[Proposition $5.1$]{BCM-pleijel} from equation (37) to equation (42). Since $w=v=0$ on $\pa V\cap S_{\theta_0}(\delta)$, we have $\chi v\in H_0^1(V)$, and so the Faber-Krahn Theorem gives the required lower bound on the area of $V$ and hence also $D$.

\quad In the second case of (\ref{eq:cases}), suppose, without loss of generality that $\norm{u_4}_{L^2(D\cap \Omega_R)}\geq \tfrac{1}{16}$. Since the boundary of $\Omega_R$ consists of a vertical line in the support of $u_4$, we can set $\theta_0=\pi$, $v=u_4$, and $V=\Omega_R\cap \text{supp}\,u_4$, so that \eqref{eqn:BCM-corner1} holds. Therefore, we can apply the same argument as in the first case to get the lower bound on the area of $D$.

\quad Therefore, in all cases, we get a lower bound on the area of $D$ in terms of $\lambda$, the fixed $\delta$, and the geometric constants given before Lemma \ref{lem:Partition}, and this concludes the proof of the lemma.
\end{proof}

\quad Next, we use Lemma \ref{lem:BCM}, together with Theorems \ref{thm:EigVar} and \ref{thm:EigFnVar}, to compare the structure of the nodal sets of $\varphi^{\eps,e}$ and $\varphi^{\eps,o}$ in $\Omega_L$ and $\Omega_R$ with those of $\phi^e$ and $\phi^o$.

\begin{lemma} \label{lem:nodal-bulk}
For each $\delta>0$, let $\Omega_L(\phi^e,\delta)$ be the subset of $\Omega_0^\delta(\eps)$ given by
\begin{align*}
    \left\{x\in\Omega_L: \emph{dist}(x,\pa\Omega_L)>\delta \emph{ and } \emph{dist}(x,N_{\phi^e})>\delta\right\}.
\end{align*}
Here $N_{\phi^e}$ is the nodal set of $\phi^e$, separating $\Omega_L$ into $\eta_e\eqdef\llbracket \phi^e,\Omega_L\rrbracket$ nodal domains. Then, there exists a constant $\delta_0>0$ such that for each $0<\delta<\delta_0$, there exists $\eps_0>0$, so that
\begin{enumerate}
\item[(i)] $\Omega_L(\phi^e,\delta)$ consists of $\eta_e$ connected components;
 \item[(ii)] for $0<\eps<\eps_0$, $\varphi^{\eps,e}$ has no nodal domain strictly contained within the region $\Omega_L\backslash\Omega_L(\phi^e,\delta)$;
    \item[(iii)] for $0<\eps<\eps_0$, $\varphi^{\eps,e}$ has the same sign as $\phi^e$ in $\Omega_L(\phi^e,\delta)$.
   
\end{enumerate}
Moreover, if the nodal set of $\phi^e$ contains no crossings, then $\delta<\delta_0$ and $\eps_0>0$ can be additionally chosen so that 
\begin{align*}
        \left\{x\in\Omega_L:\emph{dist}(x,N_{\phi^e})\leq\delta\right\}
    \end{align*}
has $\eta_e-1$ connected components. For $0<\eps<\eps_0$, the boundary of each component in the interior of $\Omega_L$ consists of two curves, with $\varphi^{\eps,e}>0$ on one of these curves and $\varphi^{\eps,e}<0$ on the other.

    The analogous statements also hold on $\Omega_R$ and also for $\varphi^{\eps,o}$ and $\phi^o$.
\end{lemma}

\quad This result states that, upon small enough perturbations of the neck, the nodal sets of $\varphi^{\eps,e}$ and $\varphi^{\eps,o}$ within $\Omega_L\cup\Omega_R$ must live either in a neighborhood of the nodal sets of $\phi^e$ and $\phi^o$ respectively or near the boundary $\p\Omega_L\cup\pa\Omega_R$. This allows us to bound the number of nodal domains produced by the perturbed eigenfunctions in a subset $\Omega_L$ bounded away from $p_0$ or analogously a subset of $\Omega_R$ bounded away from $p_1$.

\begin{proof}[Proof of Lemma \ref{lem:nodal-bulk}.] 
We prove the lemma for $\varphi^{\eps,e}$ and $\phi^e$ in $\Omega_L$, but the same proof applies both for the eigenfunctions $\varphi^{\eps,o}$ and $\phi^o$, and in $\Omega_R$. The function $\phi^e$ is an eigenfunction in $\Omega_L$, independent of $\eps$, with $\eta_e$ nodal domains. Therefore, there exists a constant $\delta_0$, depending on $\phi^e$ but independent of $\eps$ such that property \textit{(i)} holds for all $0<\delta<\delta_0$.

\quad Since the area of $\Omega_L\backslash\Omega_L(\phi^e,\delta)$ is bounded by a constant multiple of $\delta$, by reducing $\delta_0$ if necessary and choosing $\eps_0$ as in Lemma \ref{lem:BCM}, we have that for all $0<\delta<\delta_0$ and $0<\eps<\eps_0$, property \textit{(ii)} holds.

\quad Fixing $\delta$ with $0<\delta<\delta_0$, there exists a constant $c_1 = c_1(\delta)$, independent of $\eps$, such that $|\phi^e|>c_1$ in $\Omega_L(\phi^e,\delta)$. The difference $\varphi^{\eps,e}-\phi^e$ satisfies
\begin{align*}
    \Delta(\varphi^{\eps,e}-\phi^e) = -\mu(\varphi^{\eps,e}-\phi^e) -(\lambda^{\eps,e}-\mu)\varphi^{\eps,e},
\end{align*}
which by Theorems \ref{thm:EigVar} and \ref{thm:EigFnVar} is bounded above by a multiple of $\eps$ in $L^2(\Omega_L)$. Applying interior elliptic regularity estimates in $\Omega_L(\phi^e,\delta)$ for each fixed $\delta<\delta_0$, there exists $\eps_0>0$, depending on $c_1(\delta)$, such that for $0<\eps<\eps_0$, the eigenfunctions  $\varphi^{\eps,e}$ and $\phi^e$ have the same sign in $\Omega_L(\phi^e,\delta)$. Hence, property \textit{(iii)} holds. 

\quad If the nodal set of $\phi^e$ contains no crossings in the closure of $\Omega_L$, then it must consist of $\eta_e-1$ smooth, disjoint curves, separating $\Omega_L$ into a bipartite partition of $\eta_e$ components where $\phi^e$ is strictly positive or negative. Therefore, there exists $\delta_0$, such that for all $0<\delta<\delta_0$, the set $$\{x\in\Omega_L:\text{dist}(x,N_{\phi^e})\leq\delta\}$$ has $\eta_e-1$ connected components, and the set $\{x\in\Omega_L:\text{dist}(x, N_{\phi^e})=\delta\}$ consists of $\eta_e-1$ pairs of curves, where on each pair, $\phi^e$ is positive on one curve and negative on the other. Moreover, there exists a constant $c_2 = c_2(\delta)>0$ such that $|\phi^e|>c_2$ on $\{x\in\Omega_L:\text{dist}(x,N_{\phi^e})=\delta\}$.
We choose a $\delta>0$ for which the set $\{x\in\Omega_L:\text{dist}(x,N_{\phi^e})=\delta\}$ is a distance at least $\tfrac{1}{4}\delta$ from any vertices of $\Omega_L$. Under Assumption \ref{ass:phi}, we can also choose $\delta$ to ensure that the set $\{x\in\Omega_L:\text{dist}(x,N_{\phi^e})\leq\delta\}$ does not contain $p_0$. Because $\Delta(\varphi^{\eps,e}-\phi^e)=O(\eps)$ in $L^2(\Omega_L)$, we can use elliptic regularity of the Neumann problem up to the smooth parts of $\pa\Omega_L$ to determine the existence of $\eps_0>0$, depending on $c_2(\delta)$, such that for $0<\eps_0<\eps$, $\varphi^{\eps,e}$ and $\phi^e$ have the same sign on $\{x\in\Omega_L:\text{dist}(x, N_{\phi^e})=\delta\}$. 
\end{proof}

\quad Next, we prove that $\varphi^{\eps,e}$ and $\varphi^{\eps,o}$ share the same sign as $\phi^{\eps}$ and $\phi^o$ in the two connected components of $\Omega_4^{\delta}(\eps)$. This is accomplished by using Lemmas \ref{lem:Omega4}, \ref{lem:SameNumZeros}, and \ref{lem:BCM} to show that neither eigenfunction can have a nodal domain entirely contained within $\Omega_4^\delta(\eps)$.

\begin{lemma}\label{lem:NbhdofP0}
There exists a constant $\delta_0>0$, independent of $\eps>0$, such that for each $0<\delta<\delta_0$, we can find $\eps_0>0$ so that for $0<\eps<\eps_0$,
    \begin{enumerate}
        \item $\varphi^{\epsilon,e}$ has the same sign as $\phi^e(p_0)$ in  $\Omega_4^\delta(\epsilon)\cap \left\{x<\frac{1}{2}\right\}$;
        \item $\varphi^{\epsilon,o}$ has the same sign as $\phi^o(p_0)$ in $\Omega_4^\delta(\epsilon)\cap\{x<\frac{1}{2}\}$.
    \end{enumerate}
    The analogous statements hold in $\Omega_4^\delta(\epsilon)\cap \{x>\frac{1}{2}\}$.
\end{lemma}

\quad Note that Assumption \ref{ass:phi} guarantees that $\phi^e(p_0)$ and $\phi^o(p_0)$ are nonzero. If $N_e, N_o$ are as in Lemma \ref{lem:SameNumZeros}, then Lemma \ref{lem:NbhdofP0} allows us to conclude that the number of nodal domains of $\varphi^{\eps,e}$ in $R_\eps$ is equal to $N_e+1$ and the number of nodal domains of $\varphi^{\eps,o}$ in $R_\eps$ is equal to $N_o+1$. It also implies that the number of nodal domains of either eigenfunction in $\Omega_\eps$ is equal to the sum of the nodal domains in $R_\eps$ and in $\Omega_0=\Omega_L\cup\Omega_R$ minus the two nodal domains overcounted at either end of the neck.

\begin{proof}[Proof of Lemma \ref{lem:NbhdofP0}.]
We first fix $\delta_0>0$ so that the results of Lemmas \ref{lem:Partition} and \ref{lem:SameNumZeros} hold, and so that the area of $\Omega_4^\delta(\epsilon)\cap \Omega_L$ is smaller than the constant $c^*$ from Lemma \ref{lem:BCM}. Because $\Omega_L$ is locally flat near $p_0$, for each $0<\delta<\delta_0$ there exists $\eps_0>0$ so that for $0<\eps<\eps_0$, we can use elliptic regularity and Theorem \ref{thm:EigFnVar} to conclude that $\varphi^{\epsilon,e}$ and $\varphi^{\epsilon,o}$ have the same signs as $\phi^e(p_0)$ and $\phi^o(p_0)$ respectively along the semicircle portion of the boundary $\p\Omega_4^\delta(\epsilon)\cap \Omega_L$. 

\quad We let $\tilde{\delta}>0$ be the constant from Lemma \ref{lem:Omega4} and apply Lemma \ref{lem:SameNumZeros} with $\delta = \tilde{\delta}/2$ (reducing $\tilde{\delta}$ if necessary to ensure that $\tilde{\delta}/2<\delta_0$). Decreasing $\eps_0$ if necessary, this ensures that for $0<\eps<\eps_0$, $\varphi^{\epsilon,e}$ and $\varphi^{\epsilon,o}$ have the same signs as $\phi^e(p_0)$ and $\phi^o(p_0)$ on the portion of $\Omega_4^{\delta}(\eps)\cap\Omega_\eps$ with $x>\tilde{\delta}$.

\quad If the conclusion of Lemma \ref{lem:NbhdofP0} fails, then $\varphi^{\eps,e}$ or $\varphi^{\eps,o}$ must have a nodal domain completely contained in the portion of $\Omega_4^{\delta}(\eps)\cap\Omega_\eps$ with $x<\tilde{\delta}$. In this case, Lemma \ref{lem:BCM} implies that at least half of the $L^2$ mass of the eigenfunction restricted to this nodal domain must be contained in the portion of $\Omega_4^{\delta}(\eps)\cap\Omega_\eps$ contained in $R_\eps$. By Lemma \ref{lem:Omega4}, such a nodal domain must extend past $x=\tilde{\delta}$, giving a contradiction and completing this proof.
\end{proof}

\quad Combining Lemmas \ref{lem:nodal-bulk} and \ref{lem:NbhdofP0} allows us to obtain an upper bound on the number of nodal domains of $\varphi^{\eps,e}$ and $\varphi^{\eps,o}$ that intersect $\Omega_L$ and $\Omega_R$ in terms of $\llbracket \phi,\Omega_L\rrbracket$, the nodal domain count of $\phi$ in $\Omega_L$, as outlined in the following statement.

\begin{prop} \label{prop:Omega-L}
  Let $\eta_e$ be the number of nodal domains of $\phi^e$ in $\Omega_L$. There exists a constant $\eps_0>0$ such that for $0<\epsilon\leq\eps_0$, $\llbracket\varphi^{\epsilon,e},\Omega_L\rrbracket\leq \eta_e.$ 
  Moreover, if $\phi^e$ has no nodal crossings, then $\llbracket\varphi^{\eps,e},\Omega_L\rrbracket=\eta_e$. The analogous statements hold for $\varphi^{\epsilon,o}$ and $\phi^o$, and also in $\Omega_R$.
\end{prop}

\quad This result states that the number of nodal domains produced by an eigenfunction of (\ref{eq:PDE}) in the dumbbell domain cannot increase under small perturbations of the neck. In \cite{U76}, Uhlenbeck established the instability of nodal crossings under generic domain deformations, as, for example, quantified by \cite{BGM, L25} in the Dirichlet case. If a nodal crossing vanishes, then the number of nodal domains decreases, providing a geometric justification for the inequality in Proposition \ref{prop:Omega-L}. We find that, in the absence of nodal crossings, the number of nodal domains in $\Omega_L$ remains invariant for small $\eps>0$.

\begin{proof}[Proof of Proposition \ref{prop:Omega-L}.]
The same proof holds for $\varphi^{\epsilon,o}$ and $\phi^o$, and also in $\Omega_R$, so we focus on $\varphi^{\eps,e}$ and $\phi^e$ in $\Omega_L$. We fix $\delta$ and $\eps_0$ so that the conclusions of Lemmas \ref{lem:nodal-bulk} and \ref{lem:NbhdofP0} hold. For small enough $\eps$, if $\varphi^{\eps,e}$ has a nodal domain intersecting $\Omega_L$, then by  Lemma \ref{lem:nodal-bulk} part \textit{(ii)} and Lemma \ref{lem:NbhdofP0} it must also intersect $\Omega_L(\phi^e,\delta)$. Therefore, by Lemma \ref{lem:nodal-bulk} parts \textit{(i)} and \textit{(iii)}, $\varphi^{\eps,e}$ must have at most $\eta_e$ nodal domains intersecting $\Omega_L$.

\quad Now suppose that the nodal set of $\phi^e$ contains no crossings in the closure of $\Omega_L$. From Lemma \ref{lem:nodal-bulk}, the set $$\{x\in\Omega_L:\text{dist}(x,N_{\phi^e})\leq \delta\}$$ then consists of $\eta_e-1$ connected components, and for $0<\eps<\eps_0$, the function $\varphi^{\eps,e}$ changes from positive to negative on the two parts of the boundary of each of these components in the interior of $\Omega_L$. Therefore, the nodal set of $\varphi^{\eps,e}$ must contain at least $\eta_e-1$ smooth, disjoint curves in $\Omega_L$, and so $\varphi^{\eps,e}$ must have exactly $\eta_e$ nodal domains intersecting $\Omega_L$. 
\end{proof}

\quad The results in Lemma \ref{lem:SameNumZeros} and Proposition \ref{prop:Omega-L} collectively prove the main result of this section, Proposition \ref{prop:nodal-main}. Namely, if $\eta_e=\llbracket \phi^e,\Omega_L\rrbracket=\llbracket \phi^o, \Omega_L\rrbracket$,  then for small enough $\epsilon$, 
\begin{align*}
    \llbracket \varphi^{\eps,e},\Omega_\eps\rrbracket\leq 2\eta_e+N_e-1, \\
    \llbracket \varphi^{\eps,o},\Omega_\eps\rrbracket\leq 2\eta_e+N_o-1
\end{align*}
because neither eigenfunction has a nodal set that intersects $\Omega_4^\delta(\eps)$. Lemma \ref{lem:NumEvOd} provides the required relationships between $N_e,N_o$ and the relative position of $\mu$ against the Dirichlet spectrum of (\ref{eq:OntheInt2}). If $\phi$ has no crossings in its nodal set in $\Omega_L$, then we get the required equalities
\begin{align*}
    \llbracket \varphi^{\eps,e},\Omega_\eps\rrbracket= 2\eta_e+N_e-1,\\
    \llbracket \varphi^{\eps,o},\Omega_\eps\rrbracket = 2\eta_e+N_o-1.
\end{align*}
In particular, combining this with Corollary \ref{cor:index-main} in this case, the nodal deficiencies of the perturbed eigenfunctions over $\Omega_\eps$ are zero.


\bibliographystyle{plain}
\bibliography{Dumbbell}


\end{document}